\documentclass[11pt]{amsart}
\usepackage{fullpage}
\usepackage{graphicx}
\usepackage{amssymb}
\usepackage{epstopdf}
\usepackage{color}
\usepackage{bbm, dsfont}
\newtheorem{theorem}{Theorem}[section]
\newtheorem{lemma}[theorem]{Lemma}

\theoremstyle{definition}
\newtheorem{definition}[theorem]{Definition}
\theoremstyle{remark}
\newtheorem{remark}[theorem]{Remark}

\theoremstyle{example}

\DeclareFontFamily{OML}{rsfs}{\skewchar\font'177}
\DeclareFontShape{OML}{rsfs}{m}{n}{ <5> <6> rsfs5 <7> <8> <9>
rsfs7 <10> <10.95> <12> <14.4> <17.28> <20.74> <24.88> rsfs10 }{}
\DeclareMathAlphabet{\mathfs}{OML}{rsfs}{m}{n}

\newcommand{\BQ}{{\mathbb{Q}}}
\newcommand{\BR}{{\mathbb{R}}}

\newcommand{\BZ}{{\mathbb{Z}}}

\newcommand{\CC}{{\mathcal{C}}}

\newcommand{\CE}{{\mathcal{E}}}

\newcommand{\CI}{{\mathcal{I}}}

\newcommand{\CL}{{\mathcal{L}}}
\newcommand{\CM}{{\mathcal{M}}}

\newcommand{\CR}{{\mathcal{R}}}

\newcommand{\ind}{{\mathbbm{1}}}

\newcommand{\prob}{P}

\newcommand{\bae}{\begin{equation}\begin{aligned}}
\newcommand{\eae}{\end{aligned}\end{equation}}

\def\beq{ \begin{equation} }
\def\eeq{ \end{equation} }

\def\ep{\epsilon}

\def\square{\vcenter{\vbox{\hrule height .4pt
  \hbox{\vrule width .4pt height 5pt \kern 5pt
        \vrule width .4pt} \hrule height .4pt}}}

\def\ZZ{\mathbb{Z}}

\newcommand{\bi}{\text{Branching Interlacements}}
\newcommand{\dims}{\text{dim}_S}

\numberwithin{equation}{section} 

\begin{document}
\title{Connectivity properties of Branching Interlacements}

\author{Eviatar B. Procaccia}
\address[Eviatar B. Procaccia\footnote{Research supported by NSF grant 1407558}]{Texas A\&M University}
\urladdr{www.math.tamu.edu/~procaccia}
\email{eviatarp@gmail.com}

\author{Yuan Zhang}
\address[Yuan Zhang]{Texas A\&M University}
\urladdr{www.math.tamu.edu/~yzhang1988}
\email{yzhang1988@math.tamu.edu }

\thanks{We would like to thank Bal\'azs R\'ath for presenting the $\bi$ model at the "Random spatial processes and dynamics concentration week" at Texas A\&M on April 2016.}

\maketitle

\tableofcontents

\begin{abstract}
We consider connectivity properties of the Branching Interlacements model in $\BZ^d,~d\ge5$, recently introduced by Angel, R\'ath and Zhu \cite{angel2016branching}. Using stochastic dimension techniques we show that every two vertices visited by the branching interlacements are connected via at most $\lceil d/4\rceil$ conditioned critical branching random walks from the underlying Poisson process, and that this upper bound is sharp. In particular every such two branching random walks intersect if and only if $5\le d\le 8$. The stochastic dimension of branching random walk result is of independent interest. We additionally obtain heat kernel bounds for branching random walks conditioned on survival. 
\end{abstract}
\section{Introduction}
The model of Branching Interlacements, introduced by Angel, R\'ath and Zhu \cite{angel2016branching}, is a version of Sznitmann's Random Interlacements \cite{sznitmanvacant} composed of branching random walks on $\BZ^d,~d\ge 5$. This new model is proved to be (in \cite{angel2016branching}) the weak limit of a critical branching random walk on $\BZ^d/N\BZ^d$ conditioned to occupy $\lfloor uN^d\rfloor$ vertices in the torus. Analogous to Random Interlacements the $\bi$ can be realized as a Poisson Point Process over a space of transient trajectories in $\BZ^d,~d\ge 5$. Only here the trajectories stand for the range of an exploration processes over branching random walks. The main result of this paper is an analogue of the results in \cite{Random11,rath2012connectivity}, claiming that every two vertices visited by the Random Interlacements can be connected by $\lceil d/2\rceil$ trajectories from the underlying Poisson point process. Next we give a non formal statement of the main theorem (See Theorem \ref{thm:main} for the rigorous statement):

{\it Given that $x,y\in\BZ^d$ belong to the $\bi$ set, it is a.s. possible to find a path between $x$ and $y$ contained the in the trace left by at most $\lceil d/4 \rceil$ conditioned branching random walks from the underlying Poisson point process. More over this result is sharp in the sense that a.s. there are pairs of points in the $\bi$ which can not be connected by a path using the trace of $\lceil d/4 \rceil-1$ conditioned branching random walks from the underlying Poisson  point process. }

Throughout this paper, $C$ and $c$ will denote constants that may depend on other constant parameters such as the dimension. Their values can be different from place to place.

\section{Preliminaries}
In this section we formally define the $\bi$ model and recall the concept of Stochastic Dimension \cite{benjamini2004geometry}. 

\subsection{Branching random walk}
First we define an unconditioned critical geometric branching random walk. It can be constructed from a simple random walk on $\ZZ^d$ indexed on a critical geometric Galton-Watson tree. I.e., let $T$ be a  critical geometric Galton-Watson tree with root $\rho$, of which all the oriented edges can be listed according to the exploration of this tree, which is the depth first search over the Galton-Watson tree. Moreover, such exploration also gives us a natural order of vertices in each generation. I.e., let $v_{n,k}$ be the $k$th vertex visited by the exploration in the $n$th generation, if the population size of the $n$th generation is greater than or equal to $k$. And we can also let $S_T(n)$ be the population size of the $n$th generation.

Then for each edge $e=v\to v'$ where $v$ is the ancestor of $v'$ in $T$, there is a unique $n_T(e)=n$ such that this oriented edge is visited at the $n$th step of the exploration. Let $\{Z(n)\}_{n=1}^\infty$ be a sequence of i.i.d. random variables uniform on $\left\{\pm i_1,\pm i_2,\cdots, \pm i_d  \right\}$ where $i_1,\cdots, i_d$ are the unit basis of $\ZZ^d$. We can assign each edge $e$ with $Z(n_T(e))$. Then note that given $T$ and for any vertex $v=v_{n,k}\in T$, there is a unique sequence of vertices $v_{1,k_1(n,k)},v_{2,k_2(n,k)},\cdots, v_{n-1,k_{n-1}(n,k)}$ which gives the sequence of all the ancestors of $v_{n,k}$. And let $e^{(n,k)}_i=v_{i,k_i(n,k)}\to v_{i+1,k_{i+1}(n,k)}$, $i=0,\cdots, n-1$, where $v_{0,k_0(n,k)}=\rho$ and $v_{n,k_n(n,k)}=v_{n,k}$. Thus we can have the mapping $BRW$(branching random walk) from each vertex $v_{n,k}$ to $\ZZ^d$ as follows: 
$$
BRW(v_{n,k})=\left\{
\begin{aligned}
&0, \ \ &{\rm \ if \ } n=0\\
&\sum_{i=0}^{n-1} Z(n_T(e^{(n,k)}_i)) \ \ &{\rm \ if \ } n>0
\end{aligned}
\right..
$$
Then under this mapping, we have a critical geometric branching random walk on $\ZZ^d$ starting at 0. And from the construction above, one can immediately see that for any $n$ and $k$, given a $T$ with $S_T(n)\ge k$, the distribution of $BRW(v_{n,k})$ is the same as that of a summation of $n$ i.i.d. unit $d-$dimensional uniform jumps, which is the same as the distribution of $X_n$, where $\{X_n\}_{n=1}^\infty$ is a simple random walk on $\ZZ^d$ starting from 0. 

Next we construct a branching random walk conditioned on survival. We start from the back bone simple random walk. Let $\{X_n, n\ge 0\}$ be a simple random walk with $X_0=x$, and $\sigma-$field $\mathfs{\tilde F}_N(x)=\sigma(X_0,\cdots, X_N)$. Then let $\{\hat F(n)\}_{n=0,1,\cdots}$ and $\{\hat B(n)\}_{n=0,1,\cdots}$ be independent families of i.i.d. branching random walks starting at 0 with increment distribution $G(1/2)$. For any $n\ge 0$, let
$$
F(n)=\hat F(n)+X_n
$$
and
$$
B(n)=\hat B(n)+X_n.
$$
Then by our construction here and the construction after Lemma 1.2 in \cite{Recurrent2012}, we have that 
$$
D(x)=\{F(0), B(0), F(1), B(1),\cdots \}
$$
has the same distribution as a critical geometric branching random walk conditioned on survival starting at $x$. And in this paper, we will call such a process as a {\bf double branching random walk}. We call
$$
D_f(x)=\{F(0), F(1),\cdots \}
$$
and
$$
D_b(x)=\{B(0), B(1),\cdots \}
$$
the forward and backward part of $D(x)$. And we denote the trace of $D(x)$ by
$$
T(x)=\bigcup_{n=0}^\infty\left({\rm Trace}\left(F(n)\right)\cup {\rm Trace}\left(B(n)\right)\right).
$$
Moreover, let 
$$
\mathfs{\hat F}_{N}(x)=\sigma\left(\hat F(1), \hat B(1),\cdots, \hat F(N), \hat B(N) \right),
$$
and
$$
\mathfs{F}_{N}(x)=\sigma\left(\mathfs{\hat F}_{N}(x), \mathfs{\tilde F}_{N}(x)\right).
$$
It is easy to see that for each constant $N$, $\mathfs{\hat F}_{N}(x)$ and $\mathfs{\tilde F}_{N}(x)$ are independent, $F(n), B(n)\in \mathfs{F}_{N}(x)$ for all $n\le N$. Moreover, for 
$$
D_N(x)=\{F(0), B(0), F(1), B(1),\cdots, F(N),B(N)\}
$$
and
$$
T_N(x)=\bigcup_{n=0}^N\left({\rm Trace}\left(F(n)\right)\cup {\rm Trace}\left(B(n)\right)\right)
$$
it is easy to see that $D_N(x),T_N(x)\in \mathfs{F}_{N}(x)$.

\subsection{Construction of the Branching Interlacements}
For completeness we present the construction of Angel, R\'ath and Zhu \cite{angel2016branching}. For more details the reader is referred to their paper.

Let $W$ be the space of doubly-infinite nearest-neighbor trajectories in $\ZZ^d$ which tend to infinity as the time $n\to \pm \infty$, and $\mathfs{W}$ be the standard cylinder $\sigma$-algebra. And let 
$$
W^*=W\slash \sim, {\rm \ where \ } w\sim w'\Leftrightarrow w(\cdot)=w'(\cdot+k) \ {\rm for \ some \ } k\in \ZZ
$$
which is the space of equivalence classes of $W$ modulo time shift (equivalently define $\mathfs{W}^*$). Denote by $\pi$ the natural map from $W$ to $W^*$. For a finite set $K\subset\BZ^d$ denote by $W_K^0$, the set of trajectories $\gamma\in W$ which intersect $K$ and $\inf\{n:\gamma(n)\in K\}=0$.  Let $\prob_x$ be the measure on $W^*$ which is the law of the exploration on a double branching random walk rooted at $x$. Define $Q_K(\cdot)$ on $(W,\mathfs{W})$ by
$$
Q_K(A)=\sum_{x\in K}\prob_x[A\cap W_K^0],
$$
for any $K\subset\BZ^d$ and $A\in\mathfs{W}$. It is proved in \cite{angel2016branching} that there is a a unique $\sigma$-finite measure $\nu$ on $(W^*,\mathfs{W}^*)$ such that
$$
\ind_{W_K^*}\cdot \nu=\pi\circ Q_K,
$$
where $W_K^*=\pi(W_K^0)$. Now we can construct the underlying Poissson point process of the branching interlacements. Consider the set of point measures on $W^*\times\BR_+$
$$
\Omega=\{\omega=\sum_{i=1}^\infty\delta_{(w_i^*,u_i)}:w_i^*\in W^*, u_i\in\BR_+, \omega(W_K^*\times[0,u])<\infty,\text{ for every finite }K\subset\BZ^d\}
.$$
Now define $\prob$ as Poisson point process of intensity $\nu\times dx$ on $\Omega$, where $dx$ is the Lesbegue measure. For any given $0<u'<u$, let 
$$
\omega_{u',u}=\sum_{i=1}^\infty\delta_{w_i^*} \ind_{\{u'\le u_i<u\}}.
$$
In in the case where $u'=0$ we denote $\omega_u=\omega_{0,u}$. Now we denote the Branching Interlacements between levelsl $u'>0$ and $u>0$
$$
\CI^{u',u}=\bigcup_{\gamma\in {\rm supp}(\omega_{u',u})} \text{Trace}(\gamma)
,$$
And $\CI^u=\CI^{0,u}$.
Note that $\CI^u$ can be characterized by the following equation. For every finite $K\subset\BZ^d$,
$$
\prob[\CI^u\cap K=\emptyset]=e^{-u\widehat{\text{Cap}}(K)}
,$$
where $\widehat{\text{Cap}}(K)=\sum_{x\in K}\widehat{e}_K(x)$. Here $\widehat{e}_K(x)$ is the $\prob_x$-probability a double branching random walk rooted at $x$ belongs to $W_K^0$, and is called the branching equilibrium measure.
%

\subsection{Stochastic dimension}
In this section we recall some definitions and results from Benjamini, Kesten Peres Schramm 2004 \cite{benjamini2004geometry}, and adapt them to this paper. For $x,y\in\BZ^d$ let $\langle xy\rangle=\max\{|x-y|,1\}$. For any finite subset $W\subset \BZ^d$ and a tree $\tau$ on $W$ denote by $\langle \tau \rangle =\prod_{(x,y)\in\tau}\langle xy\rangle$. We define $\langle W\rangle=\min_\tau\langle \tau\rangle$, where the minimum is over all $|W|^{|W|-2}$ trees on the vertex set $W$. A random relation $\CR$ is a random subset $\CR\subset\BZ^d\times\BZ^d$. We write $x\CR y$ for $(x,y)\in\CR$. \begin{definition}
\label{dimension}
We say that a random relation $\CR$ has stochastic dimension $\alpha\in(0,d]$, and write $\dims(\CR)=\alpha$, if there is a constant $0<c<\infty$ such that
\begin{equation}
c\prob[c\CR y]\ge\langle xy\rangle^{\alpha-d}
\end{equation}
and
\begin{equation}
\prob[x\CR y, z\CR v]\le c\langle x y \rangle^{\alpha-d}\langle z v \rangle^{\alpha-d}+c\langle x y z v\rangle^{\alpha-d}
\end{equation}
for all $x,y,z,v\in\BZ^d$.
\end{definition}
For two random relations $\CR,\CL$ we consider the product defined by $x\CR\CL y$ if $\exists z\in\BZ^d$ such that $x\CR z$ and  $z\CL y$. Next we state two theorems proved in \cite{benjamini2004geometry}, we include them for reader convenience. 
\begin{theorem}\cite[Theorem 2.4]{benjamini2004geometry}
Let $\CL$ and $\CR$ be two independent random relations with stochastic dimensions. Then
$$
\dims(\CL\CR)=\min\{d,\dims(\CL)+\dims(\CR)\}
.$$
\end{theorem}

Note that if a relation has stochastic dimension $d$ then there is a uniform positive lower bound on the probability two vertices are in the relation. To push this to probability one, trivial tail sigma algebras are employed. 
\begin{definition}
Let $\CE$ be a random relation and $v\in\BZ^d$. We define the left, right and remote tail $\sigma$-algebras:
\begin{equation}
\begin{split}
&\mathfs{F}_\CE^L(v)=\bigcap_{K\subset\BZ^d\text{ finite}}\sigma\{v\CE x:x\notin K\},\\
&\mathfs{F}_\CE^R(v)=\bigcap_{K\subset\BZ^d\text{ finite}}\sigma\{x\CE v:x\notin K\},\\
&\mathfs{F}_\CE^{Rem}=\bigcap_{K_1,K_2\subset\BZ^d\text{ finite}}\sigma\{x\CE y:x\notin K_1,y\notin K_2\}.
\end{split}
\end{equation}
We say that $\CE$ is left (right) tail trivial if $\mathfs{F}_\CE^L(v)$ ($\mathfs{F}_\CE^R(v)$) is trivial for every $v\in\BZ^d$. We say that $\CE$ is remote tail trivial if $\mathfs{F}_\CE^{Rem}$ is trivial.
\end{definition}
With those definitions in hand we can state the 0-1 theorem for random relations.
\begin{theorem}\cite[Corollary 3.4]{benjamini2004geometry}
\label{0-1}
Let $m\ge 2$, and let $\{\CE_i\}_{i=1}^m$ be independent random relations such that $\dims(\CE_i)$ exists for all $i\le m$. Suppose that 
$$
\sum_{i=1}^m\dims(\CE_i)\ge d
,$$
in addition, $\CE_1$ is left tail trivial, $\CE_m$ is right tail trivial and $\CE_2,\ldots,\CE_{m-1}$ are remote tail trivial. Then for every $x,y\in\BZ^d$
$$
\prob[x\CE_1\CE_2\cdots\CE_m y]=1
.$$
\end{theorem}
\subsection{Statement of results}
First we formally define the random relations used in this paper. For any $u>u'\ge 0$, let 
$$
\mathcal{M}_{u',u}=\left\{(x,y)\in \ZZ^d\times \ZZ^d: \exists \gamma\in {\rm supp}(\omega_{u',u}), \ s.t. \  x,y\in \gamma\right\},
$$
be the random subset where two points both belong to one trajectory in the Poisson point process $\omega_{u',u}$, and again we denote $\mathcal{M}_{0,u}$ by $\mathcal{M}_{u}$. The main results of this paper are stated as follows: 
\begin{theorem}
\label{thm:1}
For the random relation $\CL$ and $\CR$ defined in \eqref{eq:L} and \eqref{eq:R}, and the random relation $\mathcal{M}_{u',u}$ for any $u>u'\ge 0$,
$$\dims(\CL)=\dims(\CR)=\dims(\CM_{u',u})=4.$$
\end{theorem}

\begin{theorem}
\label{thm:main}
For every $u>0$, and all $x,y\in\BZ^d$,
$$\prob\left[x\CM_u^{\lceil d/4\rceil}y|x,y\in\CI^u\right]=1.$$
In addition for every $u>0$,
$$
\prob\left[\exists x,y\in\CI^u, y\notin\{z:x\CM_u^{\lceil d/4\rceil-1}z\}\right]=1
.$$
\end{theorem}
The strategy of the proof goes as follows: In Section \ref{leftright}, we define random relations $\CL$ and $\CR$ independent to $\omega_u$, and in Sections \ref{upperbound} and \ref{sec:lowerbound} we prove all the upper and lower bounds to show that $\dims(\CL)=\dims(\CR)=\dims(\CM_{u',u})=4$, for every $u>u'\ge 0$. Next we define the random relation 
\beq
\label{random relation 1}
\CC=\CL\left( \prod_{i=2}^{\lceil d/4\rceil-1} \CM_{\frac{u(i-1)}{\lceil d/4\rceil},\frac{ui}{\lceil d/4\rceil}}\right)\CR.
\eeq
Following the condition of Theorem \ref{0-1}, we prove in Section \ref{leftright} that $\CL$ is left tail trivial and $\CR$ is right tail trivial. And in Section \ref{remote} we prove that $\CM_{u',u}$ is remote tail trivial. That concludes by Theorem \ref{0-1}, that 
\beq
\label{eq:C}
P\left[x\CC y \right]=1. 
\eeq
In Section \ref{main theorem} by stochastic domination and the same argument as in \cite{Random11} we prove Theorem \ref{thm:main}.

\section{Left and Right Tail Trivialities}
\label{leftright}

  Here we consider the left and right tail triviality problems of a double branching random walk starting at point $x\in \ZZ^d$ conditioned on the backward part never returning to $x$. For all $x\in \ZZ^d$, let $\{\gamma(x), x\in \ZZ^d\}$ be the trace of independent double branching random walks starting at $x$, conditioned on the backward part never returning to $x$. Then we can define the following random subsets of $\ZZ^d\times \ZZ^d$:
\beq
\label{eq:L}
\mathcal{L}=\{(x,y)\in \ZZ^d, y\in \gamma(x)\}
\eeq
and 
\beq
\label{eq:R}
\mathcal{R}=\{(x,y)\in \ZZ^d, x\in \gamma(y)\}. 
\eeq
Moreover, we can introduce the notation $x\mathcal{L}y$ which is equivalent to $(x,y)\in \mathcal{L}$, and $x\mathcal{R}y$ which is equivalent to $(x,y)\in \mathcal{R}$, in terms of random relations. 

  The result we want to prove in this section is: 
\begin{lemma}
\label{lemma left 1}
The random relation $\mathcal{L}$ is left tail trivial and the random relation $\mathcal{R}$ is right tail trivial. 
\end{lemma}
\begin{proof}
By symmetry, we can without loss of generality concentrate on the left tail triviality of the random relation $\mathcal{L}$. We will show this by first proving the left tail triviality when we no longer conditioned on that backward part never returning to $x$. For all $x\in \ZZ^d$, let $\{T(x), x\in \ZZ^d\}$ be the trace of independent double branching random walks starting at $x$. Then we can define the following random subsets of $\ZZ^d\times \ZZ^d$:
$$
\mathcal{L}^0=\{(x,y)\in \ZZ^d, y\in T(x)\}
$$
and 
$$
\mathcal{R}^0=\{(x,y)\in \ZZ^d, x\in T(y)\}. 
$$
And we have the following lemma:
\begin{lemma}
\label{lemma left 1.5}
The random relation $\mathcal{L}^0$ is left tail trivial and the random relation $\mathcal{R}^0$ is right tail trivial. 
\end{lemma}
\begin{proof}
In order to show the left tail triviality in the lemma above, we first show the following result:
\begin{lemma}
\label{lemma left 2}
For any $N\ge 0$, any event $B\in \mathfs{F}_{N}(x)$, with $P(B)>0$ which has the form $B=\hat B\cap \tilde B$, where $\hat B\in \mathfs{\hat F}_{N}(x)$ and $\tilde B\in \mathfs{\tilde F}_{N}(x)$, and any event $A\in \mathfs{F}_{\mathcal{L}^0}^L(x)$, we always have $A$ and $B$ are independent.  
\end{lemma}
\begin{proof}
To show the lemma above, it suffices to show that for any $\ep>0$
$$
|P(A\cap B)-P(A)P(B)|<\ep.
$$
Consider the sigma-field for any $r>0$:
$$
\mathfs{F}_{\mathcal{L}^0}^{L,r}(x)=\sigma\left\{\ind_{x\mathcal{L}^0 y}: y\in B(x,r)^c\right\}.
$$
It is easy to see that $\mathfs{F}_{\mathcal{L}^0}^{L,r}(x)\supset \mathfs{F}_{\mathcal{L}^0}^{L}(x)$ for any $r$. Moreover, for any subset $S\subset B(x,r)^c$, we have 
$$
\{T(x)\cap S=\O\}\in \mathfs{F}_{\mathcal{L}^0}^{L,r}(x).
$$
We will first show that 
\begin{lemma}
\label{lemma left 3}
For any $N\ge 0$, any event $B\in \mathfs{F}_{N}(x)$, with $P(B)>0$ which has the form $B=\hat B\cap \tilde B$, where $\hat B\in \mathfs{\hat F}_{N}(x)$ and $\tilde B\in \mathfs{\tilde F}_{N}(x)$, and any $\ep>0$, there is a $r_2<\infty$ such that for any event $A\in \mathfs{F}_{\mathcal{L}^0}^{L,r_2}(x)$, we have 
$$
|P(A\cap B)-P(A)P(B)|\le \ep. 
$$
\end{lemma}
\begin{proof}
To find the $r_2$ we need in the lemma, we first consider a smaller $r_1$ that will be specified later in the proof, and 
$$
\tau_{r_1}(x)=\min\{n: X_n\in \partial B(x,r_1)\}
$$
which is a stopping time with respect to $\mathfs{\tilde F}_{N}(x)$ and $\mathfs{F}_{N}(x)$. If we look at the trace after its back bone random walk first hits $\partial B(x,r_1)$, i.e.
$$
T^{r_1}=\bigcup_{n=\tau_{r_1}(x)}^\infty \left({\rm Trace}\left(F(n)\right)\cup {\rm Trace}\left(B(n)\right)\right)
$$
and given $X_{\tau_{r_1}(x)}=x'\in \partial B(x,r_1)$, then it is a double branching random walk, starting at $x'$. Moreover, given $X_{\tau_{r_1}(x)}=x'\in \partial B(x,r_1)$, $T^{r_1}$ and $\mathfs{F}_{\tau_{r_1}(x)-1}(x)$ are conditionally independent. Then for $r_2>r_1$, define family of events 
$$
\Pi=\left\{\{T(x)\cap S=\O\}, S\subset  B(x,r_2)^c\right\}
$$
which is a $\pi-$field, where $\sigma(\Pi)=\mathfs{F}_{\mathcal{L}^0}^{L,r_2}(x)$. For any $A=\{T(x)\cap S=\O\}\in \Pi$, let event 
$$
A_{r_1}=\{T^{r_1}\cap S=\O\}
$$
which is measurable with respect to $\sigma(T^{r_1})$. It is easy to see that 
$$
A_{r_1}\cap \{T_{\tau_{r_1}(x)-1}(x)\cap B(x,r_2)^c=\O\}=A\cap \{T_{\tau_{r_1}(x)-1}(x)\cap B(x,r_2)^c=\O\}.
$$
Consider the family of event 
\begin{align*}
\mathfs{D}_{r_1,x}=&\left\{A\in \mathfs{F}_{\mathcal{L}^0}^{L,r_2}(x), \ \exists A_{r_1}\in \sigma(T^{r_1}) \ s.t. \ A_{r_1}\cap \{T_{\tau_{r_1}(x)-1}(x)\cap B(x,r_2)^c=\O\}\right.\\
&\hspace{3 in} =A\cap \{T_{\tau_{r_1}(x)-1}(x)\cap B(x,r_2)^c=\O\} \Big\}.
\end{align*}
Then it is easy to see that $\Omega\in \mathfs{D}_{r_1,x}$ and $\Pi\subset \mathfs{D}_{r_1,x}$. Moreover for any $A, \hat A\in \mathfs{D}_{r_1,x}$, $A\supset \hat A$
\beq
\label{sigma1}
\begin{aligned}
(A\cap \hat A^c)\cap \{T_{\tau_{r_1}(x)-1}(x)\cap B(x,r_2)^c=\O\}&=[A\cap \{T_{\tau_{r_1}(x)-1}(x)\cap B(x,r_2)^c=\O\}]\\
&\hspace{0.5 in}\cap [\hat A \cap \{T_{\tau_{r_1}(x)-1}(x)\cap B(x,r_2)^c=\O\}]^c\\
&=[A_{r_1}\cap \{T_{\tau_{r_1}(x)-1}(x)\cap B(x,r_2)^c=\O\}]\\
& \hspace{0.5 in}\cap [\hat A_{r_1} \cap \{T_{\tau_{r_1}(x)-1}(x)\cap B(x,r_2)^c=\O\}]^c\\
&=(A_{r_1}\cap \hat A_{r_1}^c)\cap \{T_{\tau_{r_1}(x)-1}(x)\cap B(x,r_2)^c=\O\}
\end{aligned}
\eeq
which implies that $A\cap \hat A^c\in \mathfs{D}_{r_1,x}$. And for any increasing sequence $A^{(n)}\uparrow \bar A$, we have 
\beq
\label{sigma2}
\begin{aligned}
\bar A \cap \{T_{\tau_{r_1}(x)-1}(x)\cap B(x,r_2)^c=\O\}&=\bigcup_{n=1}^\infty [A^{(n)}_{r_1}\cap \{T_{\tau_{r_1}(x)-1}(x)\cap B(x,r_2)^c=\O\}]\\
&=\left[ \bigcup_{n=1}^\infty A^{(n)}_{r_1}\right]\cap \{T_{\tau_{r_1}(x)-1}(x)\cap B(x,r_2)^c=\O\}
\end{aligned}
\eeq
which implies that $\bar A\in \mathfs{D}_{r_1,x}$. Thus by $\pi-\lambda$ Theorem, $\mathfs{D}_{r_1,x}=\mathfs{F}_{\mathcal{L}^0}^{L,r_2}(x)$, which implies that for any $A\in \mathfs{F}_{\mathcal{L}^0}^{L,r_2}(x)$ there is a $A_{r_1}\in  \sigma(T^{r_1})$ such that 
$$
A_{r_1}\cap \{T_{\tau_{r_1}(x)-1}(x)\cap B(x,r_2)^c=\O\}=A\cap \{T_{\tau_{r_1}(x)-1}(x)\cap B(x,r_2)^c=\O\}.
$$
and that 
$$
|P(A_{r_1})-P(A)|\le P\left(T_{\tau_{r_1}(x)-1}(x)\cap B(x,r_2)^c\not=\O \right).
$$
However, we have the event 
\begin{align*}
\left\{T_{\tau_{r_1}(x)-1}(x)\cap B(x,r_2)^c\not=\O \right\}=&\bigcup_{n=0}^\infty\bigcup_{z\in B(x,r_1-1)}\{X_n=z\}\cap \{\hat F_n\cap B(x,r_2)^c-z\not=\O\}\cap \{\tau_{r_1}(x)>n\}\\
\cup &\bigcup_{n=0}^\infty\bigcup_{z\in B(x,r_1-1)}\{X_n=z\}\cap \{\hat B_n\cap B(x,r_2)^c-z\not=\O\}\cap \{\tau_{r_1}(x)>n\}
\end{align*}
which implies that 
$$
P\left(T_{\tau_{r_1}(x)-1}(x)\cap B(x,r_2)^c\not=\O \right)\le 2\sum_{z\in B(x,r_1-1)} P(\hat Y\cap \{B(x,r_2)^c-z\}\not=\O)\left(\sum_{n=0}^\infty P(X_n=z)\right)
$$
where $\hat Y$ is any critical branching random walk starting at 0. Note that for simple random walk $\{X_n\}_{n=0}^\infty$ starting at $x$, there exists a constant $c<\infty$ such that 
$$
\sum_{n=0}^\infty P(X_n=x)\le c
$$
and 
$$
\sum_{n=0}^\infty P(X_n=z)\le c|x-z|^{-d+2}. 
$$
Moreover, for $r_2>r_1$, and any $z\in B(x,r_1)$, noting that the random of $\hat Y$ is smaller than or equal to the population of the corresponding Galton-Watson tree, we have 
\beq
\label{upperbound r}
P(\hat Y\cap \{B(x,r_2)^c-z\}\not=\O)\le P(|\hat Y|\ge r_2-r_1)\le \frac{1}{r_2-r_1}.
\eeq
Thus there exists a $C<\infty$ such that 
$$
P\left(T_{\tau_{r_1}(x)-1}(x)\cap B(x,r_2)^c\not=\O \right)\le \frac{C r_1^2}{r_2-r_1}
$$
which implies that for any $r_1$ and $r_2>4\ep^{-1}C r_1^2+r_1$, we have $P\left(T_{\tau_{r_1}(x)-1}(x)\cap B(x,r_2)^c\not=\O \right)\le \ep/4$. Noting that 
$$
|P(A\cap B)-P(A_{r_1}\cap B)|< \ep/4
$$
and that 
$$
|P(A) P(B)-P(A_{r_1})P(B)|< \ep/4,
$$
it is sufficient to prove that for a sufficiently large $r_1$ and sufficiently larger $r_2>8\ep^{-1}C r_1^2+r_1$, we always have 
\beq
\label{left 1}
|P(A_{r_1}\cap B)-P(A_{r_1})P(B)|<\ep/2.
\eeq
To show the inequality above, note that for any $r_1>N+1$, the back bone random walk will never exit $B(x,r_1-1) $ in the first $N$ steps, which implies that $B\in \mathfs{F}_{\tau_{r_1}(x)-1}(x)$. Moreover, since that $T^{r_1}$ and $\mathfs{F}_{\tau_{r_1}(x)-1}(x)$ are conditionally independent, given $X_{\tau_{r_1}(x)}=x'\in \partial B(x,r_1)$, and that $A_{r_1}$ is measurable with respect to $\sigma(T^{r_1})$, we have 
$$
P(A_{r_1}\cap B|X_{\tau_{r_1}(x)}=x')-P(A_{r_1}|X_{\tau_{r_1}(x)}=x')P(B|X_{\tau_{r_1}(x)}=x')
$$
which implies that 
\beq
\label{left 2}
|P(A_{r_1}\cap B)-P(A_{r_1})P(B)|\le P(B)\sum_{x'\in \partial B(x,r_1)} |P(X_{\tau_{r_1}(x)}=x')-P(X_{\tau_{r_1}(x)}=x'|B)|.
\eeq
To find a upper bound of the right hand side above, note that $B=\hat B\cap \tilde B$, where $\hat B\in \mathfs{\hat F}_{N}(x)$ and $\tilde B\in \mathfs{\tilde F}_{N}(x)$, and that $\mathfs{\hat F}_{N}(x)$ is independent to $\{X_n\}_{n=0}^\infty$. We have 
\beq
\label{left 3}
\sum_{x'\in \partial B(x,r_1)} |P(X_{\tau_{r_1}(x)}=x')-P(X_{\tau_{r_1}(x)}=x'|B)|=\sum_{x'\in \partial B(x,r_1)} |P(X_{\tau_{r_1}(x)}=x')-P(X_{\tau_{r_1}(x)}=x'|\tilde B)|.
\eeq
Moreover by strong Markov property, let 
$$
\tau_{2N}(x)=\min\{n: X_n\in \partial B(x,2N)\}.
$$
We have 
\beq
\label{left 4}
\begin{aligned}
&\sum_{x'\in \partial B(x,r_1)} |P(X_{\tau_{r_1}(x)}=x')-P(X_{\tau_{r_1}(x)}=x'|\tilde B)|\\
\le&\sum_{y'\in \partial B(x,2N)}P(X_{\tau_{2N}(x)}=y'|\tilde B)\sum_{x'\in \partial B(x,r_1)} |P(X_{\tau_{r_1}(x)}=x')-P(X_{\tau_{r_1}(x)}=x'|X_{\tau_{2N}(x)}=y')|.
\end{aligned}
\eeq
Finally, by maximum random walk coupling, for any $x$ and $y'\in B(x,2N)$, we can construct $\{X^{(1,y')}_n,X^{(2,y')}_n\}_{n=0}^\infty$ in the same probability space, where $\{X^{(1,y')}_n\}_{n=0}^\infty$ is a random walk starting at $x$ and $\{X^{(2,y')}_n\}_{n=0}^\infty$ is a random walk starting at $y'$ such that 
$$
\lim_{n\to\infty} P\left(X^{(1,y')}_m=X^{(2,y')}_m, \ \forall m\ge n \right)=1. 
$$
Thus there exists a $M$ such that for any $y'\in \partial B(x,2N)$
$$
P\left(X^{(1,y')}_m=X^{(2,y')}_m, \ \forall m\ge M \right)>1-\ep/4.
$$
Then for any $r_1>M+2N$, let $\tau_{r_1,i}$, $i=1,2$ be the first time that $\{X^{(i,y')}_n\}_{n=0}^\infty$ hits the boundary of $B(x,r_1)$. We have for any $y'\in \partial B(x,2N)$
\begin{align*}
\sum_{x'\in \partial B(x,r_1)} &|P(X_{\tau_{r_1}(x)}=x')-P(X_{\tau_{r_1}(x)}=x'|X_{\tau_{2N}(x)}=y')|\\
=&\sum_{x'\in \partial B(x,r_1)}|P(X^{(1,y')}_{\tau_{r_1,1}}=x')-P(X^{(2,y')}_{\tau_{r_1,2}}=x')|\\
\le&\sum_{x'\in \partial B(x,r_1)}P(X^{(1,y')}_{\tau_{r_1,1}}=x', X^{(2,y')}_{\tau_{r_1,2}}\not=x')+\sum_{x'\in \partial B(x,r_1)}P(X^{(1,y')}_{\tau_{r_1,1}}\not=x', X^{(2,y')}_{\tau_{r_1,2}}=x')\\
\le& 2\left[1- P\left(X^{(1,y')}_m=X^{(2,y')}_m, \ \forall m\ge r_1-2N \right)\right]\\
<&\ep/2  
\end{align*}
The second inequality above is a result of the fact that $\tau_{r_1,i}>r_1-2N$ and that once we have $X^{(1,y')}_m=X^{(2,y')}_m, \ \forall m\ge r_1-2N$, this will guarantee $\tau_{r_1,1}=\tau_{r_1,2}$ and $X^{(1,y')}_{\tau_{r_1,1}}=X^{(2,y')}_{\tau_{r_1,2}}$. Thus, combining \eqref{left 2}-\eqref{left 4}, we now have \eqref{left 1} and the proof of Lemma 1.3 is complete. 

\end{proof}

  With Lemma \ref{lemma left 3}, for any $A\in \mathfs{F}_{\mathcal{L}^0}^L(x)$, noting that $\mathfs{F}_{\mathcal{L}^0}^L(x)\subset \mathfs{F}_{\mathcal{L}^0}^{L,r_2}(x)$ for any $r_2$ and let $r_2\to\infty$, we have for any event $B\in \mathfs{F}_{N}(x)$, with $P(B)>0$ which has the form $B=\hat B\cap \tilde B$, where $\hat B\in \mathfs{\hat F}_{N}(x)$ and $\tilde B\in \mathfs{\tilde F}_{N}(x)$, and any $\ep>0$,
$$
|P(A\cap B)-P(A)P(B)|<\ep
$$
which implies that $A$ and $B$ are independent. 
\end{proof}
  With Lemma  \ref{lemma left 2}, since $\mathfs{F}_N$ can be seen as the product sigma-field of $\mathfs{\tilde F}_N\times \mathfs{\hat F}_N$ and any $A\in \mathfs{F}_{\mathcal{L}^0}^L(x)$ is already independent to any cylinder event $B=\hat B\cap \tilde B$, by $\pi-\lambda$ Theorem, $A$ is independent to $\mathfs{F}_N$ for any $N\ge 0$ which implies $A$ is independent with 
$$
\sigma\left(\bigcup_{N=0}^\infty \mathfs{F}_N\right)\supset \mathfs{F}_{\mathcal{L}^0}^L(x).
$$
Thus, $A$ is independent to itself, which implies that $P(A)=0$ or 1. Thus the random relation $\mathcal{L}^0$ is left trivial.  

\end{proof}

  With the lemma above, it is then easy to adapt this result to the left triviality of the random relations $\CL$ and $\CR$. Without loss of generality, for $\CL$, note that this random relation is the same as the random relation that $y$ is in the trace of a double branching random walk starting at $x$, conditioned under event $A_0=\{$the backward part of the double branching random walk never return to $x\}$. Thus we can consider the new probability space $(A_0,P(\cdot|A_0))$ and sigma-fields:
$$
\mathfs{F}_{\mathcal{L}^0,A_0}^{L,K}(x)=\sigma(\{x\mathcal{L}^0 y\}\cap A_0: y\in K^c)
$$
for all finite $K\subset \ZZ^d$ and 
$$
\mathfs{F}_{\mathcal{L}^0,A_0}^{L}(x)=\bigcap_{K\subset \ZZ^d \ {\rm finite}}\sigma(\{x\mathcal{L}^0 y\}\cap A_0: y\in K^c).
$$
And it is sufficient to show that for any $A\in \mathfs{F}_{\mathcal{L}^0,A_0}^{L}(x)$, $P(A)=0$ or $P(A_0)$. To show this, first for any finite $K$, we have the following family of events 
$$
\Pi=\left\{\{T(x)\cap S=\O\}\cap A_0, S\subset  K^c\right\}\subset \mathfs{F}_{\mathcal{L}^0,A_0}^{L,K}(x)
$$
a $\pi-$field and $\sigma(\Pi)=\mathfs{F}_{\mathcal{L}^0,A_0}^{L,K}(x)$. Then again by $\pi-\lambda$ Theorem we have 
$$
\mathfs{F}_{\mathcal{L}^0,A_0}^{L,K}(x)= \mathfs{F}_{\mathcal{L}^0}^{L,K}(x)\cap A_0
$$
which immediately implies that 
$$
\mathfs{F}_{\mathcal{L}^0,A_0}^{L}(x)=\mathfs{F}_{\mathcal{L}^0}^{L}(x)\cap A_0.
$$
Thus for any $A\in \mathfs{F}_{\mathcal{L}^0,A_0}^{L}(x)$, $P(A)=0$ or $P(A_0)$. And the proof of Lemma \ref{lemma left 1} is complete. 
\end{proof}


\section{Remote Tail Trivialities}
\label{remote}
  In this section we prove the following lemma
\begin{lemma}
\label{lemma remote 1}
Let $u>0$. The random relation $\mathcal{M}_u$ is remote tail trivial. 
\end{lemma}
\begin{proof}
In order to show this lemma, we need to combine the idea and results in the earlier research on random interlacement without branching \cite{Random11} and the techniques we developed for left/right triviality.  For each finite subset $K\subset\subset \ZZ^d$, consider the sigma-fields
$$
\mathfs{F}_K=\sigma\left\{ x \mathcal{M}_u y: \  x,y\in K\right\}
$$
and 
$$
\mathfs{F}^{K}=\sigma\left\{ x \mathcal{M}_u y: \  x,y\in K^c\right\}.
$$
If we can prove that for any $A\in \mathfs{F}_{\mathcal{M}_u}^{Rem}$ and any $K\subset\subset \ZZ^d$, we always have that $A$ is independent to $\mathfs{F}_K$, then we will have $A$ is independent to itself which implies $P(A)=0$ or 1. Thus it is sufficient to show that 
\begin{lemma}
\label{lemma remote 2}
For any $K\subset\subset \ZZ^d$, any event $B\in \mathfs{F}_K$, and any $\ep>0$ there exists a $r_4<\infty$ such that for any event $A\in \mathfs{F}^{B(0,r_4)}$, 
$$
|P(A\cap B)-P(A)P(B)|\le2\ep. 
$$
\end{lemma}
\begin{proof} For any finite $K\subset\subset \ZZ^d$ and $r>0$ such that $B(0,r)\supset K$, we can look at the branching interlacement restricted to the trajectories that intersect $B(0,r)$. I.e., for 
$$
\omega_u=\sum_{i=1}^\infty \delta_{w_i^*}
$$ 
in the branching interlacement, we have
$$
\omega_u|_{W^*_{B(0,r)}}=\sum_{i=1}^\infty \delta_{w_i^*}\ind_{\{w_i^*\cap B(0,r)\not=\O\}}
$$
Let $\eta_{B(0,r)}=\omega_u|_{W^*_{B(0,r)}}(W^*)$ be the number of trajectories that  intersect$B(0,r)$. Similarly, we can also let 
$$
\omega_u|_{W^*_{K}}=\sum_{i=1}^\infty \delta_{w_i^*}\ind_{\{w_i^*\cap K\not=\O\}}
$$
and $\eta_{K}=\omega_u|_{W^*_{K}}(W^*)$. Then according to the definition of the Poisson Point Process, $\eta_{K}$ and $\eta_{B(0,r)}$ are Poisson random variables with parameters equal to $u{\rm \widehat{cap}}(K)$ and  $u{\rm \widehat{cap}}(B(0,r))$ respectively. And $\eta_{B(0,r)}-\eta_{K}$ is also a Poisson random variable with parameter $u[{\rm \widehat{cap}}(B(0,r))-{\rm \widehat{cap}}(K)]$. Moreover, $\omega_u|_{W^*_{B(0,r)}}$ and $\omega_u-\omega_u|_{W^*_{B(0,r)}}$ are independent point measures. 

Also, the previous research introducing the branching interlacement \cite{angel2016branching} enables us to construct $\omega_u|_{W^*_{B(0,r)}}$ directly as follows: Let 
$$
\left\{\left\{X_n(x,i)\right\}_{n=1}^\infty, i=1,2,\cdots,x\in \partial B(0,r)\right\}
$$
be independent family of simple random walks in $\ZZ^d$ with initial values $X_0(x,i)=x$. And let 
$$
\left\{\left\{\hat F(n,x,i)\right\}_{n=1}^\infty, i=1,2,\cdots,x\in \partial B(0,r)\right\}
$$
where for each $x$, $i$ and $n$, $\hat F(n,x,i)$ is an independent copy of branching random walk starting at 0, $F(n,x,i)=X_n(x,i)+\hat F(n,x,i)$, and 
$$
\left\{\left\{\hat B(n,x,i)\right\}_{n=1}^\infty, i=1,2,\cdots,x\in \partial B(0,r)\right\}
$$
where for each $x$, $i$ and $n$, $\hat B(n,x,i)$ is a independent copy of branching random walk starting at 0, $B(n,x,i)=X_n(x,i)+\hat B(n,x,i)$. Moreover, let 
$$
\{n_x, x\in \partial B(0,r)\}
$$
be i.i.d. copies of Poisson random variables with $\lambda=u$ that are independent with everything else. Then according to the construction in the introduction, we have 
$$
F(x,i)=\{F(0,x,i), F(1,x,i), F(2,x,i),\cdots\}
$$
and 
$$
B(x,i)=\{B(0,x,i), B(1,x,i), B(2,x,i),\cdots\}
$$
has the same distribution as the forward/back part of a double branching random walk. We call them the $i$th {\it forward random walk} starting at $x$ and $i$th {\it backward random walk} starting at $x$ respectively. The trace of them are denoted by 
$$
T_f(x,i)=\bigcup_{n=0}^\infty {\rm Trace} F(n,x,i)
$$
and 
$$
T_b(x,i)=\bigcup_{n=0}^\infty {\rm Trace} B(n,x,i)
$$
Moreover, we call 
$$
D(x,i)=\{F(0,x,i),B(0,x,i), F(1,x,i),B(1,x,i), F(2,x,i),B(2,x,i),\cdots\},
$$
which has the same distribution as a double branching random walk rooted at $x$, the $i$th {\it double branching random walk} starting at $x$, whose trace is denoted by  
$$
T(x,i)=T_f(x,i)\cup T_b(x,i). 
$$
Finally, for each $x$ and $i$, let 
$$
\hat T_b'(0,x,i)=\{y: y\in {\rm Trace} B(0,x,i), \  y \ {\rm is \ the \ location \ of \ an \ offspring \ of \ the \ root}\}
$$
and $T_b'(0,x,i)=x+\hat T_b'(0,x,i)$. We can define 
$$
T'_b(x,i)=T_b'(0,x,i)\cup \bigcup_{n=1}^\infty {\rm Trace} B(n,x,i)
$$
Note that for each $x$ and $i$, under the exploration of the double branching random walk, $T(x,i)\in W$ (here we will not use a new notation). Thus we can construct the branching interlacement restricted on $B(0,r)$ as 
$$
\omega_u|_{W^*_{B(0,r)}}=\sum_{x\in \partial B(0,r)}\sum_{i=1}^\infty \delta_{\pi(T(x,i))} \ind_{\{i\le n_x, T'_b(x,i)\cap B(0,r)=\O\}}, 
$$
where $\pi(\cdot)$ is the natural mapping from $W$ to $W^*$ applied to the exploration of the double random walks, which implies 
$$
\eta_{B(0,r)}=\sum_{x\in \partial B(0,r)}\sum_{i=1}^\infty \ind_{\{i\le n_x, T'_b(x,i)\cap B(0,r)=\O\}}
$$
is a Poisson random variables with parameters equal to $u{\rm \widehat{cap}}(B(0,r))$. Moreover, 
$$
\omega_u|_{W^*_{K}}=\sum_{x\in \partial B(0,r)}\sum_{i=1}^\infty \delta_{\pi(T(x,i))} \ind_{\{i\le n_x, T'_b(x,i)\cap B(0,r)=\O, T(x,i)\cap K\not=\O\}}, 
$$
and 
$$
\eta_{K}=\sum_{x\in \partial B(0,r)}\sum_{i=1}^\infty \ind_{\{i\le n_x, T'_b(x,i)\cap B(0,r)=\O, T(x,i)\cap K\not=\O\}}.
$$
Finally, for any $x$, $i$ and any $N$ which is either a constant integer or a stopping time with respect to $\left\{X_n(x,i)\right\}_{n=1}^\infty$, we can let 
$$
F_N(x,i)=\{F(0,x,i), F(1,x,i), F(2,x,i),\cdots, F(N,x,i)\},
$$
$$
B_N(x,i)=\{B(0,x,i), B(1,x,i), B(2,x,i),\cdots, B(N,x,i)\}
$$
and
$$
D_N(x,i)=\{F(0,x,i),B(0,x,i), F(1,x,i),B(1,x,i),\cdots, F(N,x,i),B(N,x,i)\}
$$
be the forward, backward and double processes {\bf until} the $N$th step of the back bone. And let 
$$
T_{N,f}(x,i)=\bigcup_{n=0}^N {\rm Trace} F(n,x,i),
$$ 
$$
T_{N,b}(x,i)=\bigcup_{n=0}^N {\rm Trace} B(n,x,i),
$$
$$
T_{N}(x,i)=T_{N,f}(x,i)\cup T_{N,b}(x,i),
$$
and 
$$
T'_{N,b}(x,i)=T_b'(0,x,i)\cup \bigcup_{n=1}^N {\rm Trace} B(n,x,i).
$$
 And we let 
$$
F^N(x,i)=\{F(N,x,i), F(N+1,x,i), F(N+2,x,i),\cdots\},
$$
$$
B^N(x,i)=\{B(N,x,i), B(N+1,x,i), B(N+2,x,i),\cdots\}
$$
and
$$
D^N(x,i)=\{F(N,x,i),B(N,x,i), F(N+1,x,i),B(N+1,x,i),\cdots\}
$$
be the forward, backward and double processes {\bf from} the $N$th step of the back bone, and  
$$
T^N_{f}(x,i)=\bigcup_{n=N}^\infty {\rm Trace} F(n,x,i),
$$ 
$$
T^N_{b}(x,i)=\bigcup_{n=N}^\infty {\rm Trace} B(n,x,i),
$$
$$
T^{N}(x,i)=T^N_{f}(x,i)\cup T^N_{b}(x,i). 
$$

  With the construction above, before we discuss the technical details of the proof, we first give an outline of the ideas we use in showing this result: 
\begin{itemize}
\item For any finite $K\subset\subset \ZZ^d$, we can consider a sufficiently large $r_1$ so that when looking at $\omega_u|_{W^*_{B(0,r_1)}}$, any event $\mathfs{F}_K$ will be almost independent to $\eta_{B(0,r_1)}$. 
\item Secondly, we can have a $r_2$ sufficiently larger than $r_1$ such that in the construction of $\omega_u|_{W^*_{B(0,r_1)}}$ above, we have with high probability all the double branching random walks will never return to $B(0,r_1)$ after their back bones first hit $\partial B(0,r_2)$. So the number of double branching random walks that survive up to reaching $\partial B(0,r_2)$ is with high probability the same as the number of those surviving forever. 
\item Then we restart those double branching random walks from $\partial B(0,r_2)$, and given an upper bound of the number of surviving copies, no matter what is the distribution of their initial values on $\partial B(0,r_2)$, there is a $r_3$ sufficiently larger than $r_2$ such that the distribution of the locations when each of the back bone random walk first hits $\partial B(0,r_3)$ are almost independent to the initial values on $\partial B(0,r_2)$. 
\item Finally, there is a $r_4$ sufficiently larger than $r_3$ such that with high probability all the branches in $\omega_u|_{W^*_{B(0,r_1)}}$ that start before the back bones exit $B(0,r_3)$ will not reach $B(0,r_4)^c$. So any $A\in \mathfs{F}^{B(0,r_4)}$ can ``almost" be determined by those double branching random walks restarted from $\partial B(0,r_3)$ and the independent trajectories in $\omega_u-\omega_u|_{W^*_{B(0,r_1)}}$. 
\end{itemize}  

  First for any given finite $K\subset\subset \ZZ^d$, and given $B\in \mathfs{F}_K$, note that $\eta_{B(0,r_1)}=\eta_{K}+[\eta_{B(0,r_1)}-\eta_{K}]$, where 
$$
[\eta_{B(0,r_1)}-\eta_{K}]\in\sigma\left(\omega_u|_{W^*_{B(0,r_1)}}-\omega_u|_{W^*_{K}}\right)
$$
which is independent to $\omega_u|_{W^*_{K}}$ and $\mathfs{F}_k$. Thus according to Lemma 3.3 of \cite{Random11}, we have that there is a $r_1$ sufficiently large such that 
$$
\sum_{n=0}^\infty |P(\eta_{B(0,r_1)}=n|B)-P(\eta_{B(0,r_1)}=n)|<\frac{\ep}{8}. 
$$
Then we construct $\omega_u|_{W^*_{B(0,r_1)}}$ using the construction described above, and let $N=\sum_{x\in \partial B(0,r_1)} n_x$ be a Poisson random variable with parameter $u|\partial B(0,r_1)|$. Then there exists a $D<\infty$ such that $P(N>D)<\ep/16$. For this given $r_1$, note that $\partial B(0,r_1)$ is finite, and we can give an order $\overline>$ to all items in it.  

  Moreover, for any $r>r_1$, $x\in \partial B(0,r_1)$ and $i\ge 1$ let 
$$
\tau_{x,i,r}=\inf\{n: X_n(x,i)\in \partial B(0,r)\}
$$
be the first time the $i$th back bone starting at $x$ hits $\partial B(0,r)$. And let 
$$
\eta_{r}=\sum_{x\in \partial B(0,r_1)}\sum_{i=1}^D \ind_{\{i\le n_x, T'_{\tau_{x,i,r},b}(x,i)\cap B(0,r)=\O\}}
$$
be (approximately) the number of trajectories in our construction that survive through thinning by its back bone first hits $\partial B(0,r)$. It is easy to see that 
$$
P\left(\eta_{r}\not= \eta_{B(0,r_1)}\right)\le P(N>D)+\sum_{x\in \partial B(0,r_1)}\sum_{i=1}^D P\left(T^{\tau_{x,i,r}}(x,i)\cap B(0,r_1)\not=\O \right).
$$
Note that given $X_{\tau_{x,i,r}}(x,i)=x'\in \partial B(0,r)$, $T^{\tau_{x,i,r}}(x,i)$ has the same distribution as the trace of a double branching random walk starting at $x'$. So according to Lemma \ref{lemma conditioned 1} there is a $c<\infty$ such that 
$$
P\left(T^{\tau_{x,i,r}}(x,i)\cap B(0,r_1)\not=\O \right)\le \frac{c|\partial B(0,r_1)|}{|r-r_1|^{d-4}}
$$
which implies there is a sufficiently large $r_2$ such that 
$$
P\left(\eta_{r_2}\not= \eta_{B(0,r_1)}\right)\le \frac{\ep}{16}+\frac{\ep}{16}=\frac{\ep}{8}. 
$$ 
Then consider the sigma-field:
\beq
\label{remote 0}
\mathfs{G}_{r_2}=\sigma\left(\sigma(\{n_x\}_{x\in \partial B(0,r_1)})\cup \bigcup_{x\in \partial B(0,r_1)}\bigcup_{i=1}^D \sigma\left(D_{\tau_{x,i,r_2}}(x,i)\right)\right).
\eeq
It is easy to see that $\eta_{r_2}\in \mathfs{G}_{r_2}$. Moreover, define
$$
E_{r_1,r_2}=\{N\le D\}\cap \bigcap_{x\in \partial B(0,r_1)}\bigcap_{i=1}^D \left\{T^{\tau_{x,i,r}}(x,i)\cap B(0,r_1)=\O \right\}.
$$
We have the following result on $\mathfs{F}_{B(0,r_1)}$. 
\begin{lemma}
\label{lemma remote 3}
For any event $B\in \mathfs{F}_{B(0,r_1)}$, there is an event $B_{r_2}\in \mathfs{G}_{r_2}$ such that 
$$
B\cap E_{r_1,r_2}=B_{r_2} \cap E_{r_1,r_2}.
$$
\end{lemma}
\begin{proof}
Consider the family of events in $\mathfs{F}_{B(0,r_1)}$ as follows: 
$$
\Pi=\left\{\bigcap_{i=1}^n\{x_i\mathcal{M}_u y_i\}: n<\infty, \ (x_i,y_i)\in B(0,r_1)\times B(0,r_1), \forall i=1,\cdots,n \right\}
$$
It is easy to see that $\Pi$ is a $\pi-$field, $\sigma(\Pi)=\mathfs{F}_{B(0,r_1)}$, and for any $B\in \Pi$, we have 
\begin{align*}
&\bigcap_{i=1}^n\{x_i\mathcal{M}_u y_i\}\cap E_{r_1,r_2}\\
=\Big(&\bigcap_{i=1}^n\bigcup_{x\in \partial B(0,r_1)}\bigcup_{i=1}^D \{n_x\ge i\}\cap\{T'_{\tau_{x,i,r},b}(x,i)\cap B(0,r_1)=\O\}\cap \{x,y\in T_{\tau_{x,i,r}}(x,i)\}\Big)\cap E_{r_1,r_2}
\end{align*}
so the lemma is satisfied in $\Pi$. Moreover, according to the same argument as in \eqref{sigma1} and \eqref{sigma2}, the family 
$$
\mathcal{D}=\{B\in \mathfs{F}_{B(0,r_1)}, \ B\cap E_{r_1,r_2}=B_{r_2} \cap E_{r_1,r_2}, \ {\rm for \ some \ } B_{r_2}\in \mathfs{G}_{r_2}\}\supset \Pi
$$
is a $\lambda-$field. The $\pi-\lambda$ Theorem finishes the proof of this lemma. 
\end{proof}
  With the lemma above we can immediately have for any $B\in \mathfs{F}_{B(0,r_1)}$,
\beq
\label{difference 1}
P(B\Delta B_{r_2})\le P(E_{r_1,r_2}^c)<\frac{\ep}{8}
\eeq
where $B\Delta B_{r_2}=\left(B\cap B_{r_2}^c\right)\cup \left(B^c\cap B_{r_2}\right)$, which implies that 
$$
|P(B)-P(B_{r_2})|\le P(E_{r_1,r_2}^c)<\frac{\ep}{8}. 
$$
And this is also true for any $B$ in the smaller sigma-field $\mathfs{F}_K$. Moreover, in the ordered set $\partial B(0,r_1)\times \{1,2,\cdots, D\}$, with the dictionary order $\widetilde>=\overline{>}\times >$, we can make the trajectories that survives until their back bones hit $\partial B(0,r_2)$ an ordered sequence. Let
$$
\vec a_1=\min\left\{(x,i): \ind_{n_x\ge i, T'_{\tau_{x,i,r_2},b}(x,i)\cap B(0,r_1)=\O}=1\right\}
$$ 
and 
$$
\vec a_k=\min\left\{(x,i) \widetilde>\vec a_{k-1}: \ind_{n_x\ge i, T'_{\tau_{x,i,r_2},b}(x,i)\cap B(0,r_1)=\O}=1\right\}
$$
for all $k\le \eta_{r_2}$. Let 
$$
\vec\xi_{r_2}=\left(X_{\tau_{\vec a_k,r_2}}(\vec a_k) \right)_{k=1}^{\eta_{r_2}}
$$
And for any $k\le \eta_{r_2}$ let 
$$
D_{r_2,k}=\left\{F(\tau_{\vec a_k,r_2},\vec a_k), B(\tau_{\vec a_k,r_2},\vec a_k), F(\tau_{\vec a_k,r_2}+1,\vec a_k), B(\tau_{\vec a_k,r_2}+1,\vec a_k),\cdots\right\}.
$$
and
$$
X_n^{r_2,k}=X_{\tau_{\vec a_k,r_2}+n}(\vec a_k).
$$
It is easy to see that given $\eta_{r_2}=n_0$ and $\vec\xi_{r_2}=\vec y$, $\big\{D_{r_2,k}\big\}_{k=1}^{\eta_{r_2}}$ has the same distribution as $n_0$ independent double branching random walks with initial values $\vec y$, with $\{X_n^{r_2,k}\}_{k=1}^{\eta_{r_2}}$ as their back bones, and they are also conditionally independent to $\mathfs{G}_{r_2}$. 

  For the $r_2$ we have above and $r_3\gg r_2$, given $\eta_{r_2}=n_0$ and $\vec\xi_{r_2}=\vec y$, consider the following stopping times with respect to $X_n^{r_2,k}$, $k\le n_0$: 
$$
\tau_{r_3,k}=\inf\left\{n: X_n^{r_2,k}\in \partial B(0,r_3)\right\}
$$
and let 
$$
\vec\gamma_{r_3}=\left(X_{\tau_{r_3,k}}^{r_2,k} \right)_{k=1}^{\eta_{r_2}}
$$
be the locations where each of those new back bone simple random walks first hit $\partial B(0,r_3)$. Then by maximum coupling theorem of random walk, see (3.18) of  \cite{Random11}, there is a $r_3$ sufficiently larger than $r_2$ such that for all $n_0\le D |\partial B(0,r_1)|$ and $\vec y\in \partial B(0,r_2)^{n_0}$
$$
\sum_{\vec x\in \partial B(0,r_3)^{n_0}} \left|P\left(\vec\gamma_{r_3}=\vec x |\eta_{r_2}=n_0, \vec\xi_{r_2}=\vec y\right)-P\left(\vec\gamma_{r_3}=\vec x |\eta_{r_2}=n_0\right) \right|<\frac{\ep}{8}. 
$$
Moreover, we can let $\big\{D_{r_2,k}^{\tau_{r_3,k}}\big\}_{k=1}^{\eta_{r_2}}$ be the double branching random walks $D_{r_2,k}$ restarted from $X_{\tau_{r_3,k}}^{r_2,k}$, which are conditionally independent to $\mathfs{G}_{r_3}$ given $\eta_{r_2}=n_0$  and $\vec\gamma_{r_3}=\vec x$, where  $\mathfs{G}_{r_3}$ is defined same as in \eqref{remote 0}

  Finally, according to \eqref{upperbound r}, for the $r_3$ we have above, we can have a $r_4$ such that with high probability any branches in the $D |\partial B(0,r_1)|$ trajectories which start before their back bones exiting $B(0,r_3)$ will never reach $\partial B(0,r_4)$. I.e., for each $x\in \partial B(0,r_1)$ and each $i\le D$, let 
$$
\tau_{r_3}(x,i)=\inf\{n: X_n(x,i)=B(0,r_3)\}
$$
and event 
$$
E_{r_3,r_4}=\bigcap_{x\in\partial B(0,r_1)}\bigcap_{i=1}^D \left\{T_{\tau_{r_3}(x,i)}(x,i)\cap \partial B(0,r_4)=\O\right\}.
$$
We have for a sufficiently large $r_4\gg r_3$, $P(E_{r_3,r_4})>1-\ep/8$. 

  At this point, we have finished the construction of $r_4$. Define sigma-field generated by the double branching random walks which survives up to $\tau_{\cdot, r_2}$ and restarting from $\partial B(0,r_3)$ above and the trajectories in $\omega_u-\omega_u|_{W^*_{B(0,r)}}$ as follows:
$$
\mathfs{H}_{r_3}=\sigma\left(\bigcup_{k=1}^{\eta_{r_2}}\sigma\left(D_{r_2,k}^{\tau_{r_3,k}}\right) \cup \sigma\left(\omega_u-\omega_u|_{W^*_{B(0,r_1)}}\right)\right).
$$
Here we denote 
$$
\bigcup_{k=1}^{\eta_{r_2}}\sigma\left(D_{r_2,k}^{\tau_{r_3,k}}\right)=\bigcup_{k=1}^{N|\partial B(0,r_1)|}\sigma\left(\bar D_{r_2,k}^{\tau_{r_3,k}}\right)
$$
where
$$
\bar D_{r_2,k}^{\tau_{r_3,k}}=\left\{
\begin{aligned}
&D_{r_2,k}^{\tau_{r_3,k}}, &\text{ if } \eta_{r_2}\ge k\\
&\O, &\text{ if } \eta_{r_2}< k
\end{aligned}
\right..
$$
Then we again have the following lemma stating that for any event $A\in \mathfs{F}^{B(0,r_4)}$, it is ``almost" also in $\mathfs{H}_{r_3}$.

\begin{lemma}
\label{lemma remote 4}
For any $A\in \mathfs{F}^{B(0,r_4)}$ there is a $A_{r_3}\in \mathfs{H}_{r_3}$ such that
$$
A\cap E_{r_1,r_2}\cap E_{r_3,r_4}=A_{r_3}\cap E_{r_1,r_2}\cap E_{r_3,r_4}
$$ 
\end{lemma} 
\begin{proof}
The proof of this lemma is similar to the previous one. Consider the family of events in $\mathfs{F}^{B(0,r_4)}$ as follows: 
$$
\Pi=\left\{\bigcap_{i=1}^n\{x_i\mathcal{M}_u y_i\}, : n<\infty, \ (x_i,y_i)\in B(0,r_4)^c\times B(0,r_4)^c, \forall i=1,\cdots,n \right\}
$$
It is easy to see that $\Pi$ is a $\pi-$field, $\sigma(\Pi)=\mathfs{F}^{B(0,r_4)}$, and for any $B\in \Pi$, we have 
\begin{align*}
&\bigcap_{i=1}^n\{x_i\mathcal{M}_u y_i\}\cap E_{r_1,r_2}\cap E_{r_3,r_4}=B_{r_3}\cap E_{r_1,r_2}\cap E_{r_3,r_4}
\end{align*}
where $B_{r_3}\in \mathfs{H}_{r_3}$ equals to 
\begin{align*}
\bigcup_{
\tiny
\begin{aligned}
A, B\subset \{1,2,\cdots n\}\\
A\cup B=\{1,2,\cdots n\}
\end{aligned}
}\left(\bigcap_{i\in A}\left\{\exists k\le \eta_{r_2}, \  x_i,y_i\in T_{r_2,k}^{\tau_{r_3,k}}\right\}\cap\bigcap_{i\in B}\left\{\exists \gamma\in {\rm supp}\left(\omega_u-\omega_u|_{W^*_{B(0,r_1)}}\right), x_i,y_i\in \gamma \right\}  \right).
\end{align*}
So the lemma is satisfied in $\Pi$. Moreover, according to the same argument as in \eqref{sigma1} and \eqref{sigma2}, the family 
$$
\mathcal{D}=\{B\in \mathfs{F}^{B(0,r_4)}, \ B\cap E_{r_1,r_2}\cap E_{r_3,r_4}=B_{r_3} \cap E_{r_1,r_2}\cap E_{r_3,r_4}, \ {\rm for \ some \ } B_{r_3}\in \mathfs{H}_{r_3}\}\supset \Pi
$$
is a $\lambda-$field. The $\pi-\lambda$ Theorem finishes the proof of this lemma. 
\end{proof}
  Same as before, the Lemma \ref{lemma remote 4} also implies that 
\beq
\label{difference 2}
P(A\Delta A_{r_3})\le P(E_{r_1,r_2}^c)+ P(E_{r_3,r_4}^c)\le \frac{\ep}{4}. 
\eeq
Moreover, for any cylinder event $A_{r_3}=\hat A_{r_3}\cap \tilde A_{r_3}\in \mathfs{H}_{r_3}$, where
$$
\hat A_{r_3}\in \sigma\left(\bigcup_{k=1}^{\eta_{r_2}}\sigma\left(D_{r_2,k}^{\tau_{r_3,k}}\right)\right)
$$
and 
$$
\tilde A_{r_3}\in \sigma\left(\omega_u-\omega_u|_{W^*_{B(0,r_1)}}\right),
$$
noting that $\omega_u-\omega_u|_{W^*_{B(0,r_1)}}$ is independent to $\omega_u|_{W^*_{B(0,r_1)}}$ and that 
$$
\sigma\left(\bigcup_{k=1}^{\eta_{r_2}}\sigma\left(D_{r_2,k}^{\tau_{r_3,k}}\right)\right)
$$
is conditionally independent to $\mathfs{G}_{r_3}$ given $\eta_{r_2}=n_0$  and $\vec\gamma_{r_3}=\vec x$, we have $A_{r_3}$ is also conditionally independent to $\mathfs{G}_{r_3}$ given $\eta_{r_2}=n_0$  and $\vec\gamma_{r_3}=\vec x$. And again by $\pi-\lambda$ Theorem, such conditional independence also holds for any $A_{r_3}\in \mathfs{H}_{r_3}$.  

  Now put everything we have together. For any $A\in  \mathfs{F}^{B(0,r_4)}$ and $B\in \mathfs{F}_K$, by \eqref{difference 1} and \eqref{difference 2}, 
\begin{align*}
|P(A\cap B)-P(A)P(B)|&\le 2P(A\Delta A_{r_3})+2P(B\Delta B_{r_2})+|P(A_{r_3}\cap B_{r_2})-P(A_{r_3})P(B_{r_2})|\\
&<\ep+|P(A_{r_3}\cap B_{r_2})-P(A_{r_3})P(B_{r_2})|. 
\end{align*}
And for $|P(A_{r_3}\cap B_{r_2})-P(A_{r_3})P(B_{r_2})|$, note that $A_{r_3}$ is conditionally independent to $\mathfs{G}_{r_3}$ given $\eta_{r_2}=n_0$  and $\vec\gamma_{r_3}=\vec x$, and that $B_{r_2}\in \mathfs{G}_{r_3}$. So we have that 
\begin{align*}
P(A_{r_3}\cap B_{r_2})&= \sum_{n=1}^{|\partial B(0,r_1)|D}\sum_{\vec x\in (\partial B(0,r_3))^{n_0}}P\left(A_{r_3}\cap B_{r_2}\cap\left\{\eta_{r_2}=n_0,\vec\gamma_{r_3}=\vec x\right\} \right)\\
&=\sum_{n=1}^{|\partial B(0,r_1)|D}\sum_{\vec x\in (\partial B(0,r_3))^{n_0}}P(A_{r_3}|\eta_{r_2}=n_0,\vec\gamma_{r_3}=\vec x)P\left(\left\{\eta_{r_2}=n_0,\vec\gamma_{r_3}=\vec x\right\} \cap B_{r_2} \right).
\end{align*}
And by total probability formula,  
$$
P(A_{r_3})=\sum_{n=1}^{|\partial B(0,r_1)|D}\sum_{\vec x\in (\partial B(0,r_3))^{n_0}}P(A_{r_3}|\eta_{r_2}=n_0,\vec\gamma_{r_3}=\vec x) P(\eta_{r_2}=n_0,\vec\gamma_{r_3}=\vec x).
$$
Thus we have
\beq
\label{remote 1}
|P(A_{r_3}\cap B_{r_2})-P(A_{r_3})P(B_{r_2})|\le {\rm Error},
\eeq
where 
$$
{\rm Error}= \sum_{n=1}^{|\partial B(0,r_1)|D}\sum_{\vec x\in (\partial B(0,r_3))^{n_0}} \left|P\left(\left\{\eta_{r_2}=n_0,\vec\gamma_{r_3}=\vec x\right\}\cap B_{r_2}\right)-P\left(\eta_{r_2}=n_0,\vec\gamma_{r_3}=\vec x\right)P(B_{r_2}) \right|.
$$
Moreover for any $n_0\le |\partial B(0,r_1)|D$ and $\vec y\in (\partial B(0,r_2))^{n_0}$, note that $\{X_n^{r_2,k}\}_{k=1}^{\eta_{r_2}}$ are conditionally independent to $\mathfs{G}_{r_2}$ given $\eta_{r_2}=n_0$ and $\vec\xi_{r_2}=\vec y$ and that $B_{r_2}\in \mathfs{G}_{r_2}$. Thus for any $n_0\le |\partial B(0,r_1)|D$ and $\vec x\in (\partial B(0,r_3))^{n_0}$, we have
\begin{align*}
&P\left(\left\{\eta_{r_2}=n_0,\vec\gamma_{r_3}=\vec x\right\}\cap B_{r_2}\right)\\
&\hspace{0.5 in} =\sum_{\vec y\in (\partial B(0,r_2))^{n_0}}P\left(\vec\gamma_{r_3}=\vec x\big| \eta_{r_2}=n_0,\vec\xi_{r_2}=\vec y\right)P\left(\left\{\eta_{r_2}=n_0,\vec\xi_{r_2}=\vec y\right\}\cap B_{r_2} \right)
\end{align*}
and 
\begin{align*}
&P(\{\eta_{r_2}=n_0\}\cap B_{r_2}) P\left(\vec\gamma_{r_3}=\vec x |\eta_{r_2}=n_0\right)\\
&=P\left(\vec\gamma_{r_3}=\vec x |\eta_{r_2}=n_0\right)\sum_{\vec y\in (\partial B(0,r_2))^{n_0}}P\left(\left\{\eta_{r_2}=n_0,\vec\xi_{r_2}=\vec y\right\}\cap B_{r_2} \right).
\end{align*}
So for any $n_0\le |\partial B(0,r_1)|D$ and $\vec x\in (\partial B(0,r_3))^{n_0}$, adding and subtracting $P(\{\eta_{r_2}=n_0\}\cap B_{r_2}) P\left(\vec\gamma_{r_3}=\vec x |\eta_{r_2}=n_0\right)$ at the same time, 
\begin{align*}
&\left|P\left(\left\{\eta_{r_2}=n_0,\vec\gamma_{r_3}=\vec x\right\}\cap B_{r_2}\right)-P\left(\eta_{r_2}=n_0,\vec\gamma_{r_3}=\vec x\right)P(B_{r_2}) \right|\\
\le&\sum_{\vec y\in (\partial B(0,r_2))^{n_0}}\left|P\left(\vec\gamma_{r_3}=\vec x |\eta_{r_2}=n_0, \vec\xi_{r_2}=\vec y\right)-P\left(\vec\gamma_{r_3}=\vec x |\eta_{r_2}=n_0\right) \right| P\left(\left\{\vec\xi_{r_2}=\vec y,\eta_{r_2}=n_0\right\}\cap B_{r_2} \right)\\
+&\left|P(B_{r_2})P(\eta_{r_2}=n_0)-P(\{\eta_{r_2}=n_0\}\cap B_{r_2}) \right|P\left(\vec\gamma_{r_3}=\vec x |\eta_{r_2}=n_0\right).
\end{align*}
Taking the summation over all $n\in \{1,2,\cdots, |\partial B(0,r_1)|D\}$ and $\vec x\in (\partial B(0,r_3))^{n_0}$, we have 
$$
{\rm Error}\le{\rm Error}_1+{\rm Error}_2.
$$
And we have 
\begin{align*}
{\rm Error}_1=&\sum_{n_0=1}^{|\partial B(0,r_1)|D}\sum_{\vec y\in (\partial B(0,r_2))^{n_0}}  P\left(\left\{\vec\xi_{r_2}=\vec y,\eta_{r_2}=n_0\right\}\cap B_{r_2} \right)\times\\
&\hspace{1 in} \sum_{\vec x\in (\partial B(0,r_3))^{n_0}}\left|P\left(\vec\gamma_{r_3}=\vec x |\eta_{r_2}=n_0, \vec\xi_{r_2}=\vec y\right)-P\left(\vec\gamma_{r_3}=\vec x |\eta_{r_2}=n_0\right) \right|\\
\le &\sum_{n_0=1}^{|\partial B(0,r_1)|D}\sum_{\vec y\in (\partial B(0,r_2))^{n_0}}  P\left(\left\{\vec\xi_{r_2}=\vec y,\eta_{r_2}=n_0\right\}\cap B_{r_2} \right)\times \frac{\ep}{8}\\
\le &\frac{\ep}{8}
\end{align*}
where the first inequality is a result of the choice of $r_3$. And
\begin{align*}
{\rm Error}_2&=\sum_{n_0=1}^{|\partial B(0,r_1)|D}\left|P(B_{r_2})P(\eta_{r_2}=n_0)-P(\{\eta_{r_2}=n_0\}\cap B_{r_2}) \right|\\
&\le \sum_{n_0=1}^{|\partial B(0,r_1)|D}\left|P(B)P(\eta_{B(0,r_1)}=n_0)-P(\{\eta_{B(0,r_1)}=n_0\}\cap B) \right|+2P(B\Delta B_{r_2})+4P(\eta_{r_2}\not=\eta_{B(0,r_1)})\\
&\le \frac{\ep}{8}+\frac{\ep}{4}+\frac{\ep}{2}=\frac{7\ep}{8}. 
\end{align*}
The second inequality is a result of the choice of $r_1$ and $r_2$. Thus we have 
$$
{\rm Error}\le{\rm Error}_1+{\rm Error}_2\le \ep
$$
and 
$$
|P(A\cap B)-P(A)P(B)|\le 2\ep. 
$$
Thus the proofs of Lemma \ref{lemma remote 1} and \ref{lemma remote 2}  are complete. 
\end{proof}

\end{proof}

%
%
%
%
%
%
%
%
%
%

\section{Upper Bounds for Branching Random Walks}
\label{upperbound}
Here we consider the upper bounds of the probabilities that the trace of a conditioned or unconditioned $d-$dimensional critical geometric branching random walk includes a certain set of cardinality 1 2 or 3.  
\subsection{Connection to One Point}
 Although the asymptotic of the probability that an unconditioned branching random walk hits one point was given in a recent research \cite{Branching2015}, we still give the lemma as follows, since the method we developed here will be useful in the discussion of the hitting probability to 2 or 3 points. For any $x\in \ZZ^d$, recalling the definition of the branching random walk, let 
$$
N_x=|\{(n,k): S_T(n)\ge k , BRW(v_{n,k})=x\}|
$$
which is the number of visits to $x$. We have the following lemma:
\begin{lemma}
\label{lemma unconditioned 1}
For $d\ge 3$, there are constant $c, C\in (0,\infty)$ such that for any $x\not=0\in \ZZ^d$
\beq
\label{unconditioned 1}
E[N_x]\in \left[c|x|^{-d+2}, C|x|^{-d+2}\right].
\eeq
Moreover, $E[N_0]<\infty$. 
\end{lemma}
\begin{proof}
Given $T=T_0$, according to the construction in the introduction, recalling that given a $T_0$ with $S_{T_0}(n)\ge k$, the distribution of $BRW(v_{n,k})$ is the same as the $n$th step of a simple random walk starting at 0, we have
\begin{align*}
E[N_x|T=T_0]&=\sum_{n=0}^\infty \sum_{k=1}^{S_{T_0}(n)} P\left(BRW(v_{n,k})=x| T=T_0\right)\\
&=\sum_{n=0}^\infty S_{T_0}(n) P(X_n=x). 
\end{align*}
Thus by the total probability theorem we have 
$$
E[N_x]=\sum_{n=0}^\infty E[S_{T}(n)] P(X_n=x),
$$
while $E[S_{T}(n)]=(E[G(1/2)]-1)E[S_{T}(n-1)]=1$ for all $n\ge 0$. Thus by the asymptotic of the Green function of simple random walk, we have 
\beq
\label{unconditioned 1 2}
E[N_x]=\sum_{n=0}^\infty  P(X_n=x)\in \left[c|x|^{-d+2}, C|x|^{-d+2}\right]
\eeq
for all $x\not=0$, and
$$
E[N_0]=\sum_{n=0}^\infty  P(X_n=0)<\infty. 
$$
\end{proof}
 As a direct result of the lemma above, we have $P(N_x>0)\le E[N_x]\le C |x|^{-d+2}$. Then for the double branching random walk, according to \cite{Recurrent2012}, we can again construct it from a critical geometric Galton-Watson tree $T_\infty$ conditioned to survive and a sequence of i.i.d. unit jumps in $\ZZ^d$.  Thus we can define 
$$
\bar N_x=|\{(n,k): S_T(n)\ge k , DBRW(v_{n,k})=x\}|
$$
where $DBRW(\cdot)$ is the mapping from the critical geometric Galton-Watson tree conditioned to survive to $\ZZ^d$. For the upper bound of $E[\bar N_x]$ we have the following result:
\begin{lemma}
\label{lemma conditioned 1}
For $d\ge 5$, there are constant $c, C\in (0,\infty)$ such that for any $x\not=0\in \ZZ^d$
$$
E[\bar N_x]\in \left[c|x|^{-d+4}, C|x|^{-d+4}\right].
$$
Moreover, $E[N_0]<\infty$. 
\end{lemma}
\begin{proof}
In order to show this lemma, from the construction of a double branching random walk in the Introduction we have for any $x\in \ZZ^d$
\begin{align*}
E[\bar N_x]&\le 2\sum_{n=0}^\infty\sum_{y\in \ZZ^d}P(X_n=y) E[N_{x-y}]\\
&\le 2\sum_{y\in \ZZ^d} E[N_{x-y}]\left(\sum_{n=0}^\infty P(X_n=y)\right)
\end{align*}
and
\begin{align*}
E[\bar N_x]&\ge \sum_{n=0}^\infty\sum_{y\in \ZZ^d}P(X_n=y) E[N_{x-y}]\\
&\ge \sum_{y\in \ZZ^d} E[N_{x-y}]\left(\sum_{n=0}^\infty P(X_n=y)\right). 
\end{align*}
Thus for $x=0$, we have 
\begin{align*}
E[\bar N_0]&\le2\sum_{y\in \ZZ^d} E[N_{-y}]\left(\sum_{n=0}^\infty P(X_n=y)\right)\\
&\le C+C\sum_{y\not=0\in \ZZ^d} |y|^{-2d+4}\\
&\le C+C \sum_{n=1}^\infty |y|^{-d+3}<\infty. 
\end{align*}
And for any $x\not=0$, we have 
\beq
\label{condition 1 1}
E[\bar N_x]\le C|x|^{-d+2}+C\sum_{y\in  \ZZ^d-\{0,x\}} |x-y|^{-d+2} |y|^{-d+2}. 
\eeq
To control the right hand side of \eqref{condition 1 1}, let $B_1=B(0,|x|/2)$, $B_2=B(x,|x|/2)$, and $B_3=B(0,2|x|)$. We have 
\beq
\label{condition 1 2}
\sum_{y\in  B_1-\{0\}} |x-y|^{-d+2} |y|^{-d+2}\le C(|x|/2)^{-d+2}\left(\sum_{n=1}^{|x|/2} n\right)\le C |x|^{-d+4},
\eeq
and 
\beq
\label{condition 1 2}
\sum_{y\in  B_2-\{x\}} |x-y|^{-d+2} |y|^{-d+2}\le C(|x|/2)^{-d+2}\left(\sum_{n=1}^{|x|/2} n\right)\le C |x|^{-d+4}. 
\eeq
Moreover, we have 
\beq
\label{condition 1 3}
\sum_{y\in  B_3-B_1-B_2} |x-y|^{-d+2} |y|^{-d+2}\le C(2|x|)^d (|x|/2)^{-2d+4}\le C |x|^{-d+4}, 
\eeq
and
\beq
\label{condition 1 4}
\sum_{y\in  B_3^c} |x-y|^{-d+2} |y|^{-d+2}\le C\left(\sum_{n=2|x|}^\infty n^{-d+3}\right)\le C |x|^{-d+4}.
\eeq
Combining all the terms we have in \eqref{condition 1 2}-\eqref{condition 1 4} gives the upper bound we want. On the other hand, note that there is a $c>0$ such that 
$$
|B_3-B_1-B_2|\ge c |x|^d
$$
and that 
$$
E[\bar N_x]\ge c\sum_{y\in  \ZZ^d-\{0,x\}} |x-y|^{-d+2} |y|^{-d+2}.
$$
Thus 
\beq
\label{condition 1 5}
E[\bar N_x]\ge c\sum_{y\in  B_3-B_1-B_2} |x-y|^{-d+2} |y|^{-d+2}\ge c(|x|)^d (2|x|)^{-2d+4}\ge c |x|^{-d+4}.
\eeq
And the proof of this lemma is complete. 
\end{proof}
 From the arguments above, it is important to note that there exists a $C<\infty$ such that for any $x\not=0$ 
\beq
\label{upper bound 1}
\sum_{y\in  \ZZ^d-\{0,x\}} |x-y|^{-d+2} |y|^{-d+2}\le C|x|^{-d+4}.
\eeq
This is an upper bound we will use again and again in the argument for 2 or 3 points.

\subsection{Connection to Two Points}

 If we have two different points $x,y\not=0$, then the following lemma controls the probability that $x$ and $y$ are both in the trace of an unconditioned critical geometric branching random walk. 
\begin{lemma}
\label{lemma uncondition 2}
There exists a $C<\infty$ such that for any $x,y\not=0, x\not=y$,
\beq
\label{uncondition 2}
\begin{aligned}
&P(N_x>0, N_y>0)\\
&\le C\left(|x|^{-d+4}|y|^{-d+2}+|y|^{-d+2}|x|^{-d+4}+|x-y|^{-d+4}|x|^{-d+2}+|x-y|^{-d+4}|y|^{-d+2} \right). 
\end{aligned}
\eeq
\end{lemma}
\begin{proof}
For any $x,y$ and any configuration of critical geometric branching random walk, noting that it is finite with probability one, we can define the {\it most recent common ancestor} of $x$ and $y$ as follows: First for any $n$ and $k$ let 
$$
T^{n,k}=\{v_{n',k'}: v_{n',k'} {\rm \ is \ an \ offspring \ of \ } v_{n,k}\}
$$
be all the offsprings of $v_{n,k}$ (and $T^{n,k}=\O$ if $k>S_T(n)$), and $B^{n,k}=BRW(T^{n,k})$ be their locations embedded to $\ZZ^d$, and $Tr^{n,k}$ be the trace of $B^{n,k}$. We can define 
$$
n_{x,y}=\sup\{n: \ \exists k\le S_T(n), \ s.t.\  \{x,y\}\subset Tr^{n,k}\}
$$
and 
$$
k_{x,y}=\sup\{k: \ s.t. \ \{x,y\}\subset Tr^{n_{x,y},k} \}.
$$
By definition, $n_{x,y}$ and $k_{x,y}$ are uniquely defined for any configuration under $\{N_x>0, N_y>0\}$. And if $n_{x,y}=0$, we call that $x$ and $y$ are {\bf separated from the root}. Moreover for any $n\ge 0,k\ge 1$ and $z\in \ZZ^d$, note that the event
$$
\{N_x>0, N_y>0, n_{x,y}=n, k_{x,y}=k, BRW(v_{n,k})=z\}
$$
is a subset of the following event:
$$
\{S_T(n)\ge k, BRW(v_{n,k})=z,  \{x,y\}\subset Tr^{n,k}, \ {\rm and \ }  \{x,y\}\nsubseteq Tr^{n',k'}, \forall v_{n',k'}\in T^{n,k}\}.
$$
In words, the event that $BRW(v_{n,k})=z$ and that $v_{n,k}$ is the most recent common ancestor of $x$ and $y$ must be included in the event that $BRW(v_{n,k})=z$, $v_{n,k}$ is a common ancestor of $x$ and $y$, and all offsprings of $v_{n,k}$ is not a common ancestor of $x$ and $y$. By doing this, we forget the fact that $v_{n,k}$ is the last particle in the $n$th generation according to the ordering of exploration/depth first search that has offsprings both at $x$ and $y$, we also forget about all the $(n+1)$th generation particles that are not an offspring of $v_{n,k}$. Thus
\beq
\label{uncondition 2 1}
\begin{aligned}
&P(N_x>0, N_y>0, n_{x,y}=n, k_{x,y}=k, BRW(v_{n,k})=z)\\
&\le P(S_T(n)\ge k, BRW(v_{n,k})=z) \\
&\hspace{0.4 in}P\left(\{x,y\}\subset Tr^{n,k}, \{x,y\}\nsubseteq Tr^{n',k'}, \forall v_{n',k'}\in T^{n,k}\big|S_T(n)\ge k, BRW(v_{n,k})=z\right).
\end{aligned}
\eeq

 Then note that given $S_T(n)\ge k$ and $BRW(v_{n,k})=z$, $B^{n,k}$ forms another critical geometric branching random walk starting at $z$, and that given $S_T(n)\ge k$ and $BRW(v_{n,k})=z$, the event $\{\{x,y\}\subset Tr^{n,k}, \{x,y\}\nsubseteq Tr^{n',k'}, \forall v_{n',k'}\in T^{n,k}\}$ is the same as the event that $x$ and $y$ are separated from the root in the new critical geometric branching random walk $B^{n,k}$. Thus we have 
\beq
\label{uncondition 2 2}
\begin{aligned}
&P\left(\{x,y\}\subset Tr^{n,k}, \{x,y\}\nsubseteq Tr^{n',k'}, \forall v_{n',k'}\in T^{n,k}\big|S_T(n)\ge k, BRW(v_{n,k})=z\right)\\
&=P(N_{x-z}>0, N_{y-z}>0, n_{x-z,y-z}=0).
\end{aligned}
\eeq
Combining the \eqref{uncondition 2 1} and  \eqref{uncondition 2 2}, we have 
\beq
\label{uncondition 2 3}
\begin{aligned}
&P(N_x>0, N_y>0, n_{x,y}=n, k_{x,y}=k, BRW(v_{n,k})=z)\\
&\le P(S_T(n)\ge k, BRW(v_{n,k})=z)P(N_{x-z}>0, N_{y-z}>0, n_{x-z,y-z}=0).
\end{aligned}
\eeq
Moreover, note that 
$$
\ind_{\{N_{x-z}>0, N_{y-z}>0, n_{x-z,y-z}=0\}} \le \sum_{i\not=j}  \ind_{\{x-z\in Tr^{1,i}\}} \ind_{\{y-z\in Tr^{1,j}\}} \ind_{S_T(1)\ge\max\{i,j\}}
$$
which implies that 
\beq
\label{uncondition 2 4}
\begin{aligned}
&P(N_{x-z}>0, N_{y-z}>0, n_{x-z,y-z}=0)\le  \sum_{i\not=j} P(S_T(1)\ge\max\{i,j\}) \\
& \hspace{2.5 in}P\left(\{x-z\in Tr^{1,i}\}\cap \{y-z\in Tr^{1,j}\}\big|S_T(1)\ge\max\{i,j\} \right).
\end{aligned}
\eeq
Then noting that given $S_T(1)\ge\max\{i,j\}$, $Tr^{1,i}$ and $Tr^{1,j}$ are the traces of two independent critical geometric branching random walk starting uniformly from the neighbors of 0, we have 
\beq
\label{uncondition 2 5 1}
\begin{aligned}
&P\left(\{x\in Tr^{1,i}\}\cap \{y\in Tr^{1,j}\}\big|S_T(1)\ge\max\{i,j\} \right)\\
&\le C \min\{1,||x-z|-1|^{-d+2}, |x-z|^{-d+2}\} \min\{1,||y-z|-1|^{-d+2}, |y-z|^{-d+2}\}\\
&\le C |x-z|^{-d+2}|y-z|^{-d+2}
\end{aligned}
\eeq
for all $z\not= x,y$. Thus, combining \eqref{uncondition 2 4} and \eqref{uncondition 2 5 1} we have 
\beq
\label{uncondition 2 5}
\begin{aligned}
P(N_{x-z}>0, N_{y-z}>0, n_{x-z,y-z}=0)&\le C E[S_T(1)^2-S_T(1)] |x-z|^{-d+2}|y-z|^{-d+2}\\
&\le C |x-z|^{-d+2}|y-z|^{-d+2}.
\end{aligned}
\eeq
\begin{remark}
We can without loss of generality simplify \eqref{uncondition 2 5 1} as above since we can always drop the finite number of terms of form $|x|^{-d+2}|x-y|^{-d+2}$ when $z=x$ or $y$, which do not have the leading order in our lemma.  
\end{remark}
 Taking the summation over all $n,k$ and $z$, and by \eqref{uncondition 2 1}, \eqref{uncondition 2 5} and Lemma \ref{lemma unconditioned 1}, we have
\beq
\label{uncondition 2 6}
\begin{aligned}
&P(N_x>0, N_y>0)\\
&=\sum_{z\in \ZZ^d} \sum_{n=0}^\infty \sum_{k=1}^\infty P(N_x>0, N_y>0, n_{x,y}=n, k_{x,y}=k, BRW(v_{n,k})=z)\\
&\le C\sum_{z\in \ZZ^d-\{x,y,0\}}\left[\sum_{n=0}^\infty \sum_{k=1}^\infty P(S_T(n)\ge k, BRW(v_{n,k})=z)\right] |x-z|^{-d+2}|y-z|^{-d+2}\\
&\le C\sum_{z\in \ZZ^d-\{x,y,0\}} |z|^{-d+2} |x-z|^{-d+2}|y-z|^{-d+2}. 
\end{aligned}
\eeq
To control the upper bound we have above, let $D=\min\{|x|,|y|\}/2$. For $z\in B(0, D)-\{0\}$ we have 
\begin{align*}
\sum_{z\in B(0, D)-\{0\}} |z|^{-d+2} |x-z|^{-d+2}|y-z|^{-d+2}&\le C\left(\sum_{1}^D n\right) (|x|-D)^{-d+2}(|y|-D)^{-d+2}\\
&\le C(|x|)^{-d+4}(|y|)^{-d+2}+C(|x|)^{-d+2}(|y|)^{-d+4}. 
\end{align*}
And for $z\in B(0, D)^c-\{x,y\}$, by \eqref{upper bound 1}
\begin{align*}
\sum_{z\in  B(0, D)^c-\{x,y\}} |z|^{-d+2} |x-z|^{-d+2}|y-z|^{-d+2}&\le D^{-d+2} \sum_{z\in  \ZZ^d-\{x,y\}}|x-z|^{-d+2}|y-z|^{-d+2}\\
&\le C D^{-d+2} |x-y|^{-d+4}\\
&\le C |x|^{-d+2} |x-y|^{-d+4}+C |y|^{-d+2} |x-y|^{-d+4}. 
\end{align*}
Combining the two inequalities above, the proof of this lemma is complete. 
\end{proof}

 And then for the double branching random walk, we have
\begin{lemma}
\label{lemma condition 2}
There exists a $C<\infty$ such that for any $x,y\not=0, x\not=y$,
\beq
\label{condition 2}
P(\bar N_x>0, \bar N_y>0)\le C\left(|x|^{-d+4}|y|^{-d+4}+|x-y|^{-d+4}|x|^{-d+4}+|x-y|^{-d+4}|y|^{-d+4} \right). 
\eeq
\end{lemma}
\begin{proof}
Recalling again the construction of the double branching random walk we have in the introduction, let 
$$
\tau_x=\inf\{n: x\in {\rm Trace}{B(n)}\cup {\rm Trace}{F(n)}\}
$$
and 
$$
\tau_y=\inf\{n: y\in {\rm Trace}{B(n)}\cup {\rm Trace}{F(n)}\}.
$$
Then it is easy to see that they are both stopping times with respect to the filtration of sigma fields  $\mathfs{F}_n=\sigma(X_0,\cdots, X_{n},\hat B(0),\cdots, \hat B(n),\hat F(0),\cdots, \hat D(n))$. Moreover, for any $n$, given $X_{n}=z$, $D^{n+1}=\{B(n+1),F(n+1), B(n+2),F(n+2),\cdots\}$ has the same distribution as a double branching random walk starting from a uniform distribution on the nearest neighbors of $z$. And $D^{n+1}$ is conditionally independent to $\mathfs{F}_n$ given $X_{n}=z$.  And since
$$
P(\bar N_x>0, \bar N_y>0)=P(\tau_x<\tau_y<\infty)+P(\tau_y<\tau_x<\infty)+P(\tau_x=\tau_y<\infty),
$$
to prove this lemma we only need to control each of  the three probabilities above. 

 For $P(\tau_x<\tau_y<\infty)$ and without loss of generality also $P(\tau_y<\tau_x<\infty)$, it is easy to see that 
\beq
\label{condition 2 1}
P(\tau_x<\tau_y<\infty)=\sum_{n=0}^\infty\sum_{z\in \ZZ^d} P(\tau_x=n<\tau_y<\infty, X_n=z).
\eeq
And that for each $n$ and $z$, we have event 
$$
\{\tau_x=n<\tau_y<\infty, X_n=z\}
$$
is a subset of event 
$$
\left\{X_n=z, x-z\in {\rm Trace}(\hat B(n))\cup {\rm Trace}(\hat F(n)),y\in {\rm Trace}(D^{n+1}) \right\}.
$$
Noting that 
$$
\left\{X_n=z, x-z\in {\rm Trace}(\hat B(n))\cup {\rm Trace}(\hat F(n))\right\}\in \mathfs{F}_n
$$
and that $\hat B(n)$ and $\hat F(n)$ is independent to $X_n$, 
\begin{align*}
P(&\tau_x=n<\tau_y<\infty, X_n=z)\\
&\le P(X_n=z)P\left(x-z\in {\rm Trace}(\hat B(n))\cup{\rm Trace}(\hat F(n))\right) P\left(y\in {\rm Trace}(D^{n+1})| X_n=z\right)\\
&\le CP(X_n=z) |x-z|^{-d+2} |y-z|^{-d+4}
\end{align*}
for all $z\not= x,y$. Thus we have 
\beq
\label{condition 2 2}
\begin{aligned}
P(\tau_x<\tau_y<\infty)&\le C \sum_{n=0}^\infty\sum_{z\in \ZZ^d-\{x,y\}} P(X_n=z) |x-z|^{-d+2} |y-z|^{-d+4}\\
&\le C \sum_{z\in \ZZ^d-\{0,x,y\}} |z|^{-d+2}  |x-z|^{-d+2} |y-z|^{-d+4}.
\end{aligned}
\eeq
To control the summation we have on the right hand side of \eqref{condition 2 2}, we again look at a neighborhood of $y$, and let $D=\min\{|y|,|x-y|\}/2$. Then
\beq
\label{condition 2 2 0}
\begin{aligned}
C\sum_{z\in B(y.D)-\{y\}}  |z|^{-d+2}  |x-z|^{-d+2} |y-z|^{-d+4}&\le C |D|^4 |y|^{-d+2}  |x-y|^{-d+2}\\
&\le C |y|^{-d+4}|x-y|^{-d+4},
\end{aligned}
\eeq
and by \eqref{upper bound 1}
\beq
\label{condition 2 2 1}
\begin{aligned}
C\sum_{z\in B(y.D)^c-\{0,x\}}|z|^{-d+2}  |x-z|^{-d+2} |y-z|^{-d+4}&\le C|D|^{-d+4} \sum_{z\in \ZZ^d-\{0,x\}}|z|^{-d+2}  |x-z|^{-d+2} \\
&\le C |y|^{-d+4} |x|^{-d+4}+ C|x-y|^{-d+4} |x|^{-d+4}. 
\end{aligned}
\eeq
Combining the two inequalities above, then we have both $P(\tau_x<\tau_y<\infty)$ and $P(\tau_y<\tau_x<\infty)$ be bounded by the form of the right hand side of \eqref{condition 2}. 

 And lastly for the part of $P(\tau_x=\tau_y<\infty)$, we again have the decomposition:
$$
P(\tau_x=\tau_y<\infty)=\sum_{n=0}^\infty\sum_{z\in \ZZ^d} P(\tau_x=\tau_y=n, X_n=z)
$$
and for each $n$ and $z$,
$$
P(\tau_x=\tau_y=n, X_n=z)\le CP(X_n=z)\left[P\left(N_{x-z}>0,N_{y-z}>0\right)+P(N_{x-z}>0)P(N_{y-z}>0)\right].
$$
Thus taking the leading order and by Lemma \ref{lemma uncondition 2} we have 
\beq
\label{condition 2 3}
\begin{aligned}
P(\tau_x=\tau_y<\infty)\le &C \sum_{z\in \ZZ^d-\{0,x,y\}} |z|^{-d+2}  |x-z|^{-d+2} |y-z|^{-d+4}\\
&+C \sum_{z\in \ZZ^d-\{0,x,y\}} |z|^{-d+2}  |x-z|^{-d+4} |y-z|^{-d+2}\\
&+C |x-y|^{-d+4}  \sum_{z\in \ZZ^d-\{0,x\}} |z|^{-d+2}|x-z|^{-d+2}\\
&+C|x-y|^{-d+4}  \sum_{z\in \ZZ^d-\{0,y\}} |z|^{-d+2}|y-z|^{-d+2}
\end{aligned}
\eeq
where the first two parts on the right hand side has been controlled above in \eqref{condition 2 2 0} and \eqref{condition 2 2 1}. Then by \eqref{upper bound 1},
$$
C |x-y|^{-d+4}  \sum_{z\in \ZZ^d-\{0,x\}} |z|^{-d+2}|x-z|^{-d+2}\le C |x-y|^{-d+4}|x|^{-d+4}
$$
and 
$$
C |x-y|^{-d+4}  \sum_{z\in \ZZ^d-\{0,y\}} |z|^{-d+2}|y-z|^{-d+2}\le C |x-y|^{-d+4}|y|^{-d+4}.
$$
Thus the probability $P(\tau_x=\tau_y<\infty)$ can also be bounded by the form of the right hand side in the lemma, and the proof of the lemma is complete. 
\end{proof}

\subsection{Connection to Three Points}

For the connection to three points $x$, $y$ and $z$, the argument is the same to the case of two points, but the notation can be very complicated. First let
$$
\begin{aligned}
&h(x,y,z)\\
&=|x|^{-d+2}|x-y|^{-d+4}|x-z|^{-d+4}+|x|^{-d+2}|x-y|^{-d+4}|y-z|^{-d+4}+|x|^{-d+2}|x-z|^{-d+4}|y-z|^{-d+4}\\
&+|y|^{-d+2}|x-y|^{-d+4}|y-z|^{-d+4}+|y|^{-d+2}|x-y|^{-d+4}|x-z|^{-d+4}+|y|^{-d+2}|y-z|^{-d+4}|x-z|^{-d+4}\\
&+|z|^{-d+2}|x-z|^{-d+4}|y-z|^{-d+4}+|z|^{-d+2}|x-z|^{-d+4}|x-y|^{-d+4}+|z|^{-d+2}|y-z|^{-d+4}|x-y|^{-d+4}\\
&+|x|^{-d+2}|y|^{-d+4}|y-z|^{-d+4}+|x|^{-d+2}|y|^{-d+4}|x-z|^{-d+4}\\
&+|x|^{-d+4}|y|^{-d+2}|y-z|^{-d+4}+|x|^{-d+4}|y|^{-d+2}|x-z|^{-d+4}\\
&+|x|^{-d+2}|z|^{-d+4}|x-y|^{-d+4}+|x|^{-d+2}|z|^{-d+4}|y-z|^{-d+4}\\
&+|x|^{-d+4}|z|^{-d+2}|x-y|^{-d+4}+|x|^{-d+4}|z|^{-d+2}|y-z|^{-d+4}\\
&+|y|^{-d+2}|z|^{-d+4}|x-y|^{-d+4}+|y|^{-d+2}|z|^{-d+4}|x-z|^{-d+4}\\
&+|y|^{-d+4}|z|^{-d+2}|x-y|^{-d+4}+|y|^{-d+4}|z|^{-d+2}|x-z|^{-d+4}\\
&+|x|^{-d+2}|y|^{-d+4}|z|^{-d+4}+|x|^{-d+4}|y|^{-d+2}|z|^{-d+4}+|x|^{-d+4}|y|^{-d+4}|z|^{-d+2}. 
\end{aligned}
$$
\normalsize
Then, for any three different $x,y,z\not=0$ and $\tau$ to be a tree on $\{0,x,y,z\}$, let $\langle\tau\rangle=\prod_{\{a,b\}\in \tau} |a-b|=\tau_1\cdot\tau_2\cdot\tau_3$, where $\tau_1, \tau_2, \tau_3$ are the lengths of the 3 edges in $\tau$, which is a product of 3 terms. 
$$
h_c(x,y,z)=\sum_{\tau} \langle\tau\rangle^{-d+4}=\sum_{\tau}\tau_1^{-d+4}\cdot\tau_2^{-d+4}\cdot\tau_3^{-d+4}
$$
which is a summations over 16 terms. Then we have for unconditioned  critical geometric branching random walk, 
\begin{lemma}
\label{lemma uncondition 3}
There exists a $C<\infty$ such that for any $x,y,z\not=0, x\not=y\not=z$,
\beq
\label{uncondition 3}
P(N_x>0,N_y>0,N_z>0)\le C h(x,y,z).
\eeq
\end{lemma}
\begin{proof}
Again we look at the most recent common ancestor of $x,y,z$ and separate the event into different cases. With the same argument as we have in the proof of Lemma \ref{lemma uncondition 2} and \ref{lemma condition 2} on more different situations, we have there is a $C<\infty$ such that the probability $P(N_x>0,N_y>0,N_z>0)$ can be bounded by $C$ times the summations of the following two types: 
$$
C|y-z|^{-d+4}\sum_{w\in \ZZ^d-\{0,x,y\}}  |w|^{-d+2}|w-x|^{-d+2}|w-y|^{-d+2}
$$
and 
$$
C\sum_{w\in \ZZ^d-\{0,x,y,z\}} |w|^{-d+2}|w-x|^{-d+4}|w-y|^{-d+2}|w-z|^{-d+2}
$$ 
where the locations of $x$, $y$ and $z$ can be permuted over all possible orders. For the first type, by Lemma \ref{lemma uncondition 2} we have 
\begin{align*}
&|y-z|^{-d+4}\sum_{w\in \ZZ^d-\{0,x,y\}}  |w|^{-d+2}|w-x|^{-d+2}|w-y|^{-d+2}\\
&\le C|y-z|^{-d+4}\left(|x|^{-d+4}|y|^{-d+2}+|y|^{-d+2}|x|^{-d+4}+|x-y|^{-d+4}|x|^{-d+2}+|x-y|^{-d+4}|y|^{-d+2} \right)\\
&\le Ch(x,y,z). 
\end{align*}
And for the second type, let $D=\min\{|x|,|x-y|,|x-z|\}/2$. And for $B(x,D)$, we have 
\begin{align*}
&\sum_{w\in B(x,D)-\{x\}} |w|^{-d+2}|w-x|^{-d+4}|w-y|^{-d+2}|w-z|^{-d+2}\\
&\le C D^4 |x|^{-d+2}|x-y|^{-d+2}|x-z|^{-d+2}\\
&\le C|x|^{-d+2}|x-y|^{-d+4}|x-z|^{-d+4}\le h(x,y,z),
\end{align*}
and by Lemma \ref{lemma uncondition 2},
\begin{align*}
&\sum_{w\in B(x,D)^c-\{0,y,z\}} |w|^{-d+2}|w-x|^{-d+4}|w-y|^{-d+2}|w-z|^{-d+2}\\
&\le C D^{-d+4} \sum_{w\in \ZZ^d-\{0,y,z\}} |w|^{-d+2}|w-y|^{-d+2}|w-z|^{-d+2}\\
&\le C \left(|x|^{-d+4}+|x-y|^{-d+4}+|x-z|^{-d+4} \right)\\
&\hspace{1 in} \times \left(|y|^{-d+4}|z|^{-d+2}+|y|^{-d+2}|z|^{-d+4}+|y-z|^{-d+4}|y|^{-d+2}+|y-z|^{-d+4}|z|^{-d+2} \right)\\
&\le Ch(x,y,z). 
\end{align*}
Thus the proof of this lemma is complete. 
\end{proof}

 And similarly, for the double critical geometric branching random walk, we can also have 
\begin{lemma}
\label{lemma condition 3}
There exists a $C<\infty$ such that for any $x,y,z\not=0, x\not=y\not=z$,
\beq
\label{condition 3}
P(\bar N_x>0,\bar N_y>0,\bar N_z>0)\le C h_c(x,y,z).
\eeq
\end{lemma}

\begin{proof}
With the same argument as we have before, we have there is a $C<\infty$ such that the probability $P(\bar N_x>0,\bar N_y>0,\bar N_z>0)$ can be bounded by $C$ times the summations of the following 5 types:
$$
C|y-z|^{-d+4}|x-z|^{-d+4}\sum_{w\in \ZZ^d-\{0,x,y\}}  |w|^{-d+2}|w-y|^{-d+2},
$$ 
$$
C|y-z|^{-d+4}|y-x|^{-d+4}\sum_{w\in \ZZ^d-\{0,x,y\}}  |w|^{-d+2}|w-y|^{-d+2},
$$ 
$$
C|y-z|^{-d+4}\sum_{w\in \ZZ^d-\{0,x,y\}}  |w|^{-d+2}|w-x|^{-d+4}|w-y|^{-d+2},
$$
$$
C|x-z|^{-d+4}\sum_{w\in \ZZ^d-\{0,x,y\}}  |w|^{-d+2}|w-x|^{-d+4}|w-y|^{-d+2},
$$
and 
$$
C\sum_{w\in \ZZ^d-\{0,x,y,z\}} |w|^{-d+2}|w-x|^{-d+4}|w-y|^{-d+4}|w-z|^{-d+2}
$$ 
where the locations of $x$, $y$ and $z$ can be permuted over all possible orders. For the first two types, we have by \eqref{upper bound 1}, 
$$
C|y-z|^{-d+4}|x-z|^{-d+4}\sum_{w\in \ZZ^d-\{0,x,y\}}  |w|^{-d+2}|w-y|^{-d+2}\le C|y-z|^{-d+4}|x-z|^{-d+4}|y|^{-d+4}\le Ch_c(x,y,z)
$$
and 
$$
C|y-z|^{-d+4}|y-x|^{-d+4}\sum_{w\in \ZZ^d-\{0,x,y\}}  |w|^{-d+2}|w-y|^{-d+2}\le C|y-z|^{-d+4}|y-x|^{-d+4}|y|^{-d+4}\le Ch_c(x,y,z).
$$
For the third and fourth type, again by Lemma \ref{lemma condition 2} we have 
\begin{align*}
&|y-z|^{-d+4}\sum_{w\in \ZZ^d-\{0,x,y\}}  |w|^{-d+2}|w-x|^{-d+4}|w-y|^{-d+2}\\
&\le C|y-z|^{-d+4}\left(|x|^{-d+4}|y|^{-d+4}+|x-y|^{-d+4}|x|^{-d+4}+|x-y|^{-d+4}|y|^{-d+4} \right)\\
&\le Ch_c(x,y,z),
\end{align*}
and 
\begin{align*}
&|x-z|^{-d+4}\sum_{w\in \ZZ^d-\{0,x,y\}}  |w|^{-d+2}|w-x|^{-d+4}|w-y|^{-d+2}\\
&\le C|x-z|^{-d+4}\left(|x|^{-d+4}|y|^{-d+4}+|x-y|^{-d+4}|x|^{-d+4}+|x-y|^{-d+4}|y|^{-d+4} \right)\\
&\le Ch_c(x,y,z).
\end{align*}
And for the fifth type, first let $D_1=\min\{|x|,|x-y|,|x-z|\}/2$. And for $B(x,D_1)$, we have 
\begin{align*}
&\sum_{w\in B(x,D_1)-\{x\}} |w|^{-d+2}|w-x|^{-d+4}|w-y|^{-d+4}|w-z|^{-d+2}\\
&\le C D^4 |x|^{-d+2}|x-y|^{-d+4}|x-z|^{-d+2}\\
&\le C|x|^{-d+4}|x-y|^{-d+4}|x-z|^{-d+4}\le h(x,y,z).
\end{align*}
Then let $D_2=\min\{|y|,|y-z|\}/2$. We have by the proof of Lemma \ref{lemma condition 2} , 
\begin{align*}
&\sum_{w\in B(x,D_1)^c\cap B(y,D_2)-\{y\}} |w|^{-d+2}|w-x|^{-d+4}|w-y|^{-d+4}|w-z|^{-d+2}\\
&\le D_1^{-d+4} \sum_{w\in B(y,D_2)-\{y\}} |w|^{-d+2}|w-y|^{-d+4}|w-z|^{-d+2}\\
&\le C  D_1^{-d+4} |y|^{-d+4}|y-z|^{-d+4}\\
&\le C\left( |x|^{-d+4}|y|^{-d+4}|y-z|^{-d+4}+ |x-y|^{-d+4}|y|^{-d+4}|y-z|^{-d+4}+ |x-z|^{-d+4}|y|^{-d+4}|y-z|^{-d+4}\right)\\
&\le Ch_c(x,y,z). 
\end{align*}
And finally, 
\begin{align*}
&\sum_{w\in B(x,D_1)^c\cap B(y,D_2)^c-\{0,z\}} |w|^{-d+2}|w-x|^{-d+4}|w-y|^{-d+4}|w-z|^{-d+2}\\
&\le D_1^{-d+4} D_2^{-d+4}\sum_{w\in \ZZ^d-\{0,z\}} |w|^{-d+2}|w-z|^{-d+2}\\
&\le CD_1^{-d+4} D_2^{-d+4}|z|^{-d+4}\\
&\le C\left(|x|^{-d+4}+|x-y|^{-d+4}+|x-z|^{-d+4}\right)\left(|y|^{-d+4}+|y-z|^{-d+4}\right)|z|^{-d+4}\\
&\le Ch_c(x,y,z). 
\end{align*}
Thus the proof of this lemma is complete. 
\end{proof}

\section{Lower Bounds for Branching Random Walks}\label{sec:lowerbound}
In this section, we find the lower bounds of the probabilities that a  $d$-dimensional double branching random walk hits one point. First, the asymptotic of the probability that an unconditioned critical geometric branching random walk hits one point was given in a recent research \cite{Branching2015} as follows: There exist $c,C\in (0,\infty)$ such that for any $x\not=0$
\beq
\label{asymptotic 1}
c |x|^{-d+2}\le P(N_x>0)\le C |x|^{-d+2}.
\eeq
And for the double branching random walk, we show that 
\begin{lemma}
\label{lemma lower bound}
There exist $c,C\in (0,\infty)$ such that for any $x\not=0$
\beq
\label{asymptotic 2}
c |x|^{-d+4}\le P(\bar N_x>0)\le C |x|^{-d+4}.
\eeq
\end{lemma}
\begin{proof}
The part of the upper bound has already been shown in Lemma \ref{lemma conditioned 1}, so we will concentrate on the lower bound. First, it is easy to see the desired result follows immediately if we can show the same lower bound for the backward branch of the double branching random walk. Then, back to the unconditioned critical geometric branching random walk, combining the lower bound in \eqref{asymptotic 1}, the upper bound in Lemma 1.1, and the fact that 
$$
E[N_x]=P(N_x>0)E[N_x|N_x>0],
$$
we immediately have, there is a $C<\infty$ such that 
$$
E[N_x|N_x>0]\le C. 
$$
Then for the backward part of the double branching random walk, let $\hat N_x$ be its number of visits to $x$, and 
$$
\hat\tau_x=\inf\{n: x\in {\rm Trace}(B(n))\}
$$
Note that $\{\hat\tau_x=n\}\subset \{\hat N_x>0\}$ and that 
\beq
\label{total expectation}
E[\hat N_x|\hat N_x>0]=\sum_{n=0}^\infty \sum_{y\in \ZZ^d} E[\hat N_x| \hat\tau_x=n, X_n=y] P( \hat\tau_x=n, X_n=y|\hat N_x>0),
\eeq
where 
$$
E[\hat N_x| \hat\tau_x=n, X_n=y]=E\left[N^{(n)}_{x-y}\big|\hat\tau_x=n, X_n=y\right]+E\left[\hat N^{(n+1)}_{x-y}\big|\hat\tau_x=n, X_n=y\right]
$$
where $N^{(n)}_{x-y}$ is the number of vertices in $\hat B(n)$ that is mapped to $x-y$, and $\hat N^{(n+1)}_{x-y}$ is the  number of vertices in $C^{n+1}=\{B(n+1),B(n+2),\cdots\}$ that is mapped to $x-y$. Note that the event 
$$
\{\hat\tau_x=n, X_n=y\}=\left\{N^{(n)}_{x-y}>0\right\}\cap A_n
$$
where 
$$
A_n=\bigcap_{i=0}^{n-1}\{{\rm Trace}(B(i))\cap x=\O\}\cap \{X_n=y\}\in \sigma \left (X_0,\cdots, X_n, \hat B(0),\cdots, \hat B(n-1)\right)
$$
which is independent to $\hat B(n)$. We have 
$$
E\left[N^{(n)}_{x-y}\big|\hat\tau_x=n, X_n=y\right]=E\left[N^{(n)}_{x-y}\big|N^{(n)}_{x-y}>0\right]\le C. 
$$
And by the fact that $C^{n+1}$ is stochastically dominated by a double critical geometric branching random walk starting from a uniformly chosen neighbor of $y$, which is conditionally independent to $\mathfs{F}_n$ given $X_n=y$, and the upper bounds we found in Lemma 1.2, we have 
$$
E\left[\bar N^{(n+1)}_{x-y}\big|\hat\tau_x=n, X_n=y\right]\le\frac{1}{2d}\sum_{k=1}^d\left( E\left[\bar N_{x-y+i_k}\right]+E\left[\bar N_{x-y-i_k}\right]\right)\le C.
$$
Thus there is a $C<\infty$ such that for any $n$, $x$ and $y$
$$
E\left[N^{(n)}_{x-y}\big|\hat\tau_x=n, X_n=y\right]\le C.
$$
Then plugging the upper bounds to the total expectation formula \eqref{total expectation}, we have $E[\hat N_x|\hat N_x>0]\le C$. Recalling the lower bound of $E[\bar N_x]$ in Lemma \ref{lemma conditioned 1}, the fact that $E[\hat N_x]>E[\bar N_x]/2$, and again the fact that 
$$
E[\hat N_x]=P(\hat N_x>0)E[\hat N_x|\hat N_x>0],
$$
We have the lower bound we need in this lemma. 
\end{proof}


\section{Proof of the Main Theorems}
\label{main theorem}
\subsection{Proof of Theorem \ref{thm:1}}
With the lemmas \ref{lemma conditioned 1}-\ref{lemma condition 3} given in Section \ref{upperbound}, for the double branching random walk starting at $x$ conditioned on $A_0=\{$the backward part of the double branching random walk never return to $x\}$. It is easy to see that 
\beq
\label{Thm 11}
P(\bar N_y>0|A_0)\le \frac{P(\bar N_y>0)}{\widehat{e}_x(x)}\le C |x-y|^{-d+4}
\eeq
and the same upper bounds holds for the probability of connection to 2 or 3 points. 

Moreover, in the branching interlacement $\CI^{u',u}$, $\forall u>u'\ge 0$, the trajectories that ever hit $x$ can be sampled by applying thinning on $N=Poisson(u)$ of double branching random walks starting at $x$. So the hitting probability is always less than or equal to that for the double branching random walks without thinning, which is bounded by $u$ times the hitting probability of each of them. Thus, the probability that there exists a trajectory in the branching random interlacement passing a certain point which also includes a set of cardinality 1, 2 or 3 has the same asymptotic results as in the lemmas \ref{lemma conditioned 1}-\ref{lemma condition 3}. 

With Lemma \ref{lemma lower bound} giving the asymptotic of the hitting probability of the double branching random walk to one point, we are also able to show the same asymptotic holds for the double branching random walk conditioned on the backward part never returning to the root. With the upper bound already shown above, consider $D=\{D(0), D(1),\cdots\}$ to be a double branching random walk starting at 0. It suffices to show that 
\beq
\label{Thm 12}
P(A_0\cap \{\bar N_x>0\})\ge c|x|^{-d+4}. 
\eeq
To show this, let stopping time 
$$
\tau=\inf\{n: X_n\in \partial B(0,|x|/2)\},
$$
where $X_n$ is the back bone simple random walk of $D$. For any $y\in \partial B(0,|x|/2)$, given $X_{\tau}=y$, $D'=\{D(\tau), D(\tau+1),\cdots\}$ is another double branching random walk starting at $y$ and is conditionally independent to $\mathfs{F}_{\tau-1}$. Moreover, let $A_0'\supset A_0$ be the event that the first $\tau-1$ backward branches never return to 0. We have $A_0'\in \mathfs{F}_{\tau-1}$. Thus 
$$
P(A_0\cap \{\bar N'_x>0\})=P(A_0'\cap \{\bar N'_x>0\})-P(A_0'\cap A_0^c\cap\{\bar N'_x>0\}),
$$
where $\bar N'$ is the number of visits in $D'$. For any $y\in \partial B(0,|x|/2)$, by Lemma \ref{lemma condition 2}
$$
P(A_0'\cap A_0^c\cap\{\bar N'_x>0\}|X_{\tau}=y)\le P(\bar N_{x-y}>0, \bar N_{-y}>0)\le C |x|^{-2d+8}=o(|x|^{-d+4}),
$$
and 
$$
P(A_0'\cap \{\bar N'_x>0\}\cap\{X_{\tau}=y\})=P(A_0'\cap\{X_{\tau}=y\}) P(\bar N_{x-y}>0)\ge cP(A_0'\cap\{X_{\tau}=y\}) |x|^{-d+4}.
$$
Combining the two inequalities above, we have the same asymptotic holds for the double branching random walk conditioned on the backward part never returning to the root. And for $\CI^{u',u}$, $\forall u>u'\ge 0$, the same asymptotic follows immediately from the fact that 
\beq
\label{Thm 13}
P[x\CM_{u',u}y]\ge P[Poisson(u-u')>0]P(A_0\cap \{\bar N_{y-x}>0\}). 
\eeq

Thus, by \eqref{Thm 11} and \eqref{Thm 12}, we have for any $x,y,z,w\in \BZ^d$
\beq
\label{dimension 1}
P(x\CL y)\ge C|x-y|^{-d+4}
\eeq
and 
\beq
\label{dimension 2}
P(x\CL y, z\CL w)\le C|x-y|^{-d+4}|z-w|^{-d+4}.
\eeq
if $x\not=z$. \eqref{dimension 2} is a result of that in $\CL$ the double branching random walks starting from each point in $\BZ^d$ are independent. And if $x=z$, we have by Lemma \ref{lemma condition 2} and the discussion above for the upper bound of the hitting probability of the double branching random walk conditioned on the backward part never returning to the root,
\beq
\label{dimension 3}
\begin{aligned}
P(x\CL y, x\CL w)&\le C\left(|x-y|^{-d+4}|x-w|^{-d+4}+|x-y|^{-d+4}|y-w|^{-d+4}+|x-w|^{-d+4}|y-w|^{-d+4}\right)\\
&= C\langle xyxw\rangle^{-d+4}.
\end{aligned}
\eeq
Combining \eqref{dimension 1}-\eqref{dimension 3}, we have by Definition \ref{dimension}, $\dims(\CL)=4$. And $\dims(\CR)=4$ follows from exactly the same argument. 

Then for the stochastic dimension of $\CM_{u',u}$, first by \eqref{Thm 13} we have 
\beq
\label{dimension 4}
P[x\CM_{u',u}y]\ge C |x-y|^{-d+4}.
\eeq
So it is sufficient for us to check the upper bound of the correlation $P(x\CM_{u',u} y, z\CM_{u',u} w)$, and the argument here is the same as the argument from (2.18) to (2.20) in \cite{Random11}. Let $K=\{x,y,z,w\}$. For $\omega_{u',u}=\sum_{i\ge 0} \delta_{w^*_i}$, we let $\hat \omega_{u',u}=\sum_{i\ge 0} \delta_{w^*_i}\ind_{\text{trace}(w^*_i)\supset K}$. So we can write 
\beq
\label{dimension 5}
P(x\CM_{u',u} y, z\CM_{u',u} w)=P(x\CM_{u',u} y, z\CM_{u',u} w,\hat \omega_{u',u}=0)+P(x\CM_{u',u} y, z\CM_{u',u} w,\hat \omega_{u',u}\not=0). 
\eeq
For a point measure $\tilde \omega\le \omega_{u',u}$, we write ``$x\CM_{u',u} y$ in $\tilde \omega$", if there is a trajectory in supp$(\tilde \omega)$ whose trace contains both $x$ and $y$. Note that for any $w^*\in \text{supp}(\omega_{u',u}-\hat \omega_{u',u})$ such that $x,y\in$Trace$(w^*)$, then at least one of $z$ or $w$ cannot belong to Trace$(w^*)$. Thus, the events $\{x\CM_{u',u} y {\rm \ in \ } \omega_{u',u}-\hat \omega_{u',u}\}$ and $\{z\CM_{u',u} w {\rm \ in \ } \omega_{u',u}-\hat \omega_{u',u}\}$ are defined in terms of disjoint sets of trajectories, and thus they are independent under the Poisson point measure $P$. So for the first term in \eqref{dimension 5}, according to such independence we have above, and the discussion on the upper bound of the probability that there exists a trajectory in the branching random interlacement passing a certain point which also includes a set of cardinality 1, we get that 
\beq
\label{dimension 6}
\begin{aligned}
P(x\CM_{u',u} y, z\CM_{u',u} w,\hat \omega_{u',u}=0)&=P(x\CM_{u',u} y {\rm \ in \ } \omega_{u',u}-\hat \omega_{u',u}, z\CM_{u',u} w {\rm \ in \ } \omega_{u',u}-\hat \omega_{u',u},\hat \omega_{u',u}=0)\\
&\le P(x\CM_{u',u} y {\rm \ in \ } \omega_{u',u}-\hat \omega_{u',u}, z\CM_{u',u} w {\rm \ in \ } \omega_{u',u}-\hat \omega_{u',u})\\
&\le P(x\CM_{u',u} y {\rm \ in \ } \omega_{u',u}-\hat \omega_{u',u})P( z\CM_{u',u} w {\rm \ in \ } \omega_{u',u}-\hat \omega_{u',u})\\
&\le P(x\CM_{u',u} y)P(z\CM_{u',u} w)\\
&\le C|x-y|^{-d+4}|z-w|^{-d+4}. 
\end{aligned}
\eeq
Then for the second term in \eqref{dimension 5}, note that 
$$
P(x\CM_{u',u} y, z\CM_{u',u} w,\hat \omega_{u',u}\not=0)= P(\hat \omega_{u',u}\not=0)
$$
and that the event $\{\hat \omega_{u',u}\not=0\}$ is the same as the event 
$$
\{\text{there exists a trajectory in } \omega_{u',u} \text{ passing } x \text{ whose trace also includes } y,z,w\},
$$
and that the discussion on the upper bound of the probability that there exists a trajectory in the branching random interlacement passing a certain point which also includes a set of cardinality 3. We have 
\beq
\label{dimension 7}
P(x\CM_{u',u} y, z\CM_{u',u} w,\hat \omega_{u',u}\not=0)\le C\langle xyzw \rangle^{-d+4}.
\eeq
Combining \eqref{dimension 4}-\eqref{dimension 7}, we have shown that $\dims(\CM_{u',u})=4$ and the proof of Theorem \ref{thm:1} is complete. 

\subsection{Proof of upper bound in Theorem \ref{thm:main}}
We follow Section 4 of \cite{Random11}. Let $A_1$ be the event that $x\in\CI^{u/\lceil d/4\rceil}$ and $A_2$ is the event that $y\in\CI^{(\lceil d/4\rceil-1)u/\lceil d/4\rceil,u}$. Conditioned on $A_1$ one gets that $\omega_{u/\lceil d/4\rceil}(W^*_{x})\ge 1$. Thus there is at least one double branching random walk emanating from $x$, conditioned on the backward part never returning to $x$. We can conclude that conditioned on $A_1$, $\{z:x\CM_{0,u/\lceil d/4\rceil}z\}$ stochastically dominates $\{z:x\CL z\}$. By the same reasoning, conditioned on $A_2$ we have that $\{z:z\CM_{(\lceil d/4\rceil-1)u/\lceil d/4\rceil,u}y\}$ stochastically dominates $\{z:z\CR y\}$. Denote $A=A_1\cap A_2$,
$$
\prob\left[x\CM_u^{\lceil d/4\rceil}y\bigg|A\right]\ge \prob\left[x\prod_{i=1}^{\lceil d/4\rceil}\CM_{u(i-1)\lceil d/4\rceil,ui\lceil d/4\rceil}y\bigg|A\right]\ge\prob[x\CC y]=1
.$$
Now for every disjoint intervals $I_1=[t_1,t_2]$ and $I_2=[t_3,t_4]$, define $A_{I_1}=\{x\in\CI^{t_1,t_2}\}$ and $A_{I_2}=\{x\in\CI^{t_3,t_4}\}$. By similar arguments one can get
\beq\label{eq:condoneprob}
\prob\left[x\CM_u^{\lceil d/4\rceil}y\bigg|A_{I_1}\cap A_{I_2}\right]=1
.\eeq
Since $$\{x,y\in\CI^u\}=\{x\CM^u y\}\cup\bigcup_{I_1,I_2 \subset [0,u],\text{ disjoint}}\{x\in\CI^{t_1,t_2},y\in\CI^{t_3,t_4}\},$$
where all $t_1,t_2,t_3,t_4\in\BQ$ are distinct. By \eqref{eq:condoneprob}, conditioned on any event in the countable positive probability union above we have $x\CM_u^{\lceil d/4\rceil}y$ a.s. Thus we conclude that
$$
\prob\left[x\CM_u^{\lceil d/4\rceil}y|x,y\in\CI^u\right]=1
.$$
\subsection{Proof of lower bound in Theorem \ref{thm:main}}
This part follows immediately from Section 5 of \cite{Random11}. The only change is in the definition of $m$ which is $m=\lceil d/4\rceil-1$ for the purpose of this result, and the stochastic dimension of the relation $\CM$ is 4 instead of 2. 
\bibliography{eigen}
\bibliographystyle{plain}

\end{document}